\documentclass[12pt]{amsart}
\usepackage{pict2e,amsmath,amsthm,amsfonts,amssymb}
\usepackage[alphabetic]{amsrefs}
\usepackage[all]{xy}
\usepackage{amssymb,amsthm}
\usepackage{eucal,mathrsfs}
\usepackage[bbgreekl]{mathbbol}
\usepackage{color,enumerate}
\usepackage{pxfonts,fullpage}
\usepackage{times}

\theoremstyle{plain}

\newtheorem{theorem}{Theorem}[section]

\newtheorem{corollary}[theorem]{Corollary}

\newtheorem{proposition}[theorem]{Proposition}

\newtheorem{lemma}[theorem]{Lemma}

\theoremstyle{definition}\newtheorem{defi}[theorem]{Definition}
\theoremstyle{remark}\newtheorem{remark}[theorem]{Remark}

\newcommand{\defpar}{\tau}

\newcommand\gmm{\Bbb{G}_m}

\newcommand\frg{\mathfrak{g}}
\newcommand\frh{\mathfrak{h}}

\newcommand\mm{\mathcal{M}}

\newcommand\mf{\mathcal{F}}

\newcommand\mk{\mathcal{K}}
\newcommand\pone{\Bbb{P}^1}

\newcommand\Pic{\operatorname{Pic}}

\newcommand\me{\mathcal{E}}

\newcommand\mv{\mathcal{V}}
\newcommand\tensor{\otimes}
\newcommand\ml{\mathcal{L}}

\newcommand\E{\mathcal{E}}

\newcommand\codim{\operatorname{codim}}

\newcommand\im{\operatorname{im}}

\newcommand\Pc{\mathcal{P}}

\newcommand\mq{\mathcal{Q}}
\newcommand\rk{\operatorname{rk}}

\newcommand\mcl{\mathcal{L}}

\newcommand\mar{\mathcal{R}}

\newcommand\Gr{\operatorname{Gr}}
\newcommand\SL{\operatorname{SL}}
\newcommand\pardeg{\operatorname{pardeg}}

\newcommand\Bun{\operatorname{Bun}}
\newcommand{\leto}[1]{\stackrel{#1}{\to}}

\newcommand\qdot{\circledast}

\newcommand{\Parbun}{\operatorname{Parbun}}

\DeclareMathOperator{\Exp}{Exp}

\DeclareMathOperator{\Ad}{Ad}

\newcommand{\Ker}{\operatorname{Ker}}

\newcommand{\fb}{\mathfrak{b}}
 \newcommand{\fg}{\mathfrak{g}}
 \newcommand{\fh}{\mathfrak{h}}
 \newcommand{\fl}{\mathfrak{l}}
 
 \newcommand{\fp}{\mathfrak{p}}

 \newcommand{\fu}{\mathfrak{u}}

\newcommand{\bz}{\mathbb{Z}}
\newcommand{\bD}{\mathbb{D}}
\def\u.{^{\bullet}}

\newtheorem{example}[theorem]{\bf Example}

\begin{document}

\title[Deformation]{The multiplicative eigenvalue problem and deformed quantum cohomology}

\author{Prakash Belkale and Shrawan Kumar}
\begin{abstract}
We construct deformations of the small quantum cohomology rings of homogeneous spaces  $G/P$, and obtain
an irredundant set of inequalities determining the multiplicative eigenvalue problem for the compact form $K$
of $G$.

\end{abstract}

\address{Department of Mathematics, University of North Carolina, Chapel Hill, NC 27599-3250}
\email{belkale@email.unc.edu, shrawan@email.unc.edu}

\maketitle
\section{Introduction}
Let $G$ be a  simple, connected, simply-connected  complex algebraic group. We choose a Borel subgroup
$B$ and a maximal torus $H\subset B$  and let $W$ be the associated Weyl group. Let $P$ be a standard parabolic subgroup (i.e., $P\supset B$) and let $L\subset P$ be its Levi subgroup containing $H$.  Then, $B_L:=B\cap L$ is a  Borel subgroup of $L$. We denote the Lie algebras of $G,P, L, B, B_L, H$ by the corresponding Gothic characters: $\fg, \fp, \fl, \fb, \fb_L, \frh$ respectively.
Let $W^P$ be the set of the minimal length representatives
in the cosets of $W/W_P$, where $W_P$ is the Weyl group of $P$. For any $w\in W^P$, let
$X_w^P:=\overline{BwP/P}\subset G/P$ be the corresponding Schubert variety and
let
$\{\sigma^P_w\}_{w\in W^P}$ be the Poincar\'e dual  (dual to the fundamental class of $X_w^P$) basis of $H^*(G/P, \mathbb{Z})$.

Let $R=R_\fg \subset \frh^*$ be
the set of roots of $\fg$ and let $R^+$ be the
 set of positive roots (i.e., the set of roots of $\fb$).
Similarly, let $R_\fl$ be the set of roots of $\fl$  and
$R_\fl^+$ be the set of roots of $\fb_L$. Let  $\Delta = \{\alpha_1, \dots, \alpha_\ell\}
\subset R^+$ be the set of simple
roots, $\{\alpha_1^\vee, \dots, \alpha_\ell^\vee\}
\subset \fh$ the
corresponding simple coroots.
 We denote by $\Delta_P$ the set of simple roots
 contained in $R_\fl$ and we set
$$ S_P:=\Delta\setminus \Delta_P.$$   For any $ 1\leq
j\leq \ell$, define the element $x_j\in \fh$
by
$$\alpha_i(x_{j})=\delta_{i,j},\text{ }\forall\text{ } 1\leq i\leq \ell.
$$

Consider the {\em fundamental alcove} $\mathscr{A}\subset
\mathfrak{h}$ defined by
$$
\mathscr{A}=\left\{\mu \in \mathfrak{h}:\alpha_{i}(\mu)\geq 0\text{~
  and~ } \theta(\mu)\leq 1\right\},
$$
where $\theta$ is the highest root of $\mathfrak{g}$. Then,
$\mathscr{A}$ parameterizes the $K$-conjugacy classes of $K$ under the
map $C:\mathscr{A}\to K/\Ad K$,
$$
\mu \mapsto c(\Exp (2\pi i\mu)),
$$
where $K$ is a maximal compact subgroup of $G$ and $c(\Exp (2\pi
i\mu))$ denotes the $K$-conjugacy class of $\Exp(2\pi i\mu)$. Fix a
positive integer $n\geq 2$ and define the {\em multiplicative polytope}
$$
\mathscr{C}_{n} := \left\{(\mu_{1},\ldots,\mu_{n})\in \mathscr{A}^{n}:
1\in C(\mu_{1})\dots C(\mu_{n})\right\}.
$$

Then,  $\mathscr{C}_{n}$ is a rational convex
polytope with nonempty interior in $\mathfrak{h}^n$.
Our aim is to describe the facets (i.e., the codimension one
faces) of $\mathscr{C}_{n}$.

We begin with the following theorem. It was proved  by  Biswas \cite{Bi} in the case $G=\SL_2$; by Belkale \cite{B1} for $G=\SL_m$ (and in this case a slightly weaker result by Agnihotri-Woodward \cite{AW} where the inequalities were parameterized by $\langle \sigma^{P}_{u_{1}},\dots,
\sigma^{P}_{u_{n}} \rangle_{d}\neq 0$); and by Teleman-Woodward \cite{Teleman-W} for general $G$. It may be recalled that the precursor to these results was the result due to Klyachko \cite{klyachko} determining the additive eigencone for $\SL_m$.

\begin{theorem}
Let $(\mu_{1},\ldots,\mu_{n})\in \mathscr{A}^{n}$. Then, the following
are equivalent:

(a) $(\mu_{1},\ldots,\mu_{n})\in \mathscr{C}_{n}$,

(b) For any standard maximal parabolic
subgroup $P$ of $G$, any $u_{1},\ldots,u_{n}\in W^{P}$, and any $d\geq
0$ such that the Gromov-Witten invariant
(cf. Definition \ref{mansur})
$$\langle \sigma^{P}_{u_{1}},\dots,
\sigma^{P}_{u_{n}} \rangle_{d}=1,
$$
the following inequality is satisfied:
$$
\mathscr{I}^{P}_{(u_{1},\ldots, u_{n};d)}:\medskip\medskip\medskip\,\,\,\,\,\,\,\,\,\,\sum^{n}_{k=1}
\omega_P(u_k^{-1}\mu_k) \leq d,
$$
where $\omega_{P}$ is the fundamental weight $\omega_{i_{P}}$ such
that $\alpha_{i_P}$ is the unique simple root in $S_{P}$.
\end{theorem}

Even though this result describes the inequalities determining the polytope $\mathscr{C}_{n}$, however for groups other than $G$ of type $A_\ell$, the above system of inequalities has redundancies. The aim of our work is to give an irredundant subsystem of inequalities determining the polytope $\mathscr{C}_{n}$.

To achieve this, similar to the notion of Levi-movability of Schubert varieties in $X=G/P$ introduced in \cite{Belkale-Kumar} which gives rise to a deformed product in the cohomology $H^*(X)$, we have introduced here the notion of {\it quantum Levi-movability}  resulting into a deformed  product in  the quantum cohomology $QH^*(X)$ parameterized by  $\{\tau_i\}_{ \alpha_i\in S_P}$ as follows.
As a $\bz[q,\tau]$-module, it is the same as $H^*(X,\bz)\otimes_\bz \bz[q,\tau]$, where $q$ (resp. $\tau$) stands for
multi variables $\{q_i\}_ {\alpha_i\in S_P}$ (resp. $\{\tau_i\}_ {\alpha_i\in S_P}$). For $u,v\in W^P$, define the $\bz[q,\tau]$-linear
{\it quantum deformed product} by
$$\sigma^P_{u}\qdot\sigma^P_{v}=\sum_{d\geq 0\in H_2(X,\bz); w\in W^P} \Bigl(\prod_{\alpha_i \in S_P} \tau_i^{A_i(u,v,w,d)}\Bigr)q^d\langle \sigma^P_u, \sigma^P_v,\sigma^P_w\rangle_d \sigma^P_{w_o w w_o^P},$$
where $w_o$ (resp. $w_o^P$) is the longest element of $W$ (resp. $W_P$),
 $$A_i(u,v,w,d) =(\chi_e-\chi_u-\chi_v-\chi_w)(x_i)+\frac{2 a_ig^*}{\langle\alpha_i,\alpha_i\rangle},$$

$\chi_w= \sum_{\beta\in (R^+\setminus R^+_\fl)\cap w^{-1} R^+} \beta$, $g^*$ is the dual Coxeter number of $\fg$ and $a_i$ is defined by the identity \eqref{defD}. It is shown that, for a  cominuscule maximal parabolic subgroup $P$,  the deformed product coincides with the original product in the quantum cohomology of $X$ (cf. Lemma \ref{minuscule}).

Evaluating each $\tau_i=0$ in the above (which is well defined because of Theorem \ref{one}), we get
$$\sigma^P_{u}\qdot_0\sigma^P_{v}=\sum'_{d, w} q^d\langle \sigma^P_u, \sigma^P_v,\sigma^P_w\rangle_d \sigma^P_{w_o w w_o^P},$$
where the sum is restricted over those $d\geq 0 \in H_2(X,\bz)$ and $w\in W^P$  so that $A_i(u,v,w,d)=0$ for all $\alpha_i\in S_P$.
We shall denote the coefficient of $q^d \sigma^P_{w_o w w_o^P}$ in  $\sigma^P_{u}\qdot_0\sigma^P_{v}$ by $\langle \sigma^P_u, \sigma^P_v,\sigma^P_w\rangle_d^{\qdot_0}.$ Similarly, we shall denote the coefficient of $q^d \sigma^P_{w_o u_n w_o^P}$ in  $\sigma^P_{u_1}\qdot_0\dots \qdot_0\sigma^P_{u_{n-1}}$ by $\langle \sigma^P_{u_1}, \dots, \sigma^P_{u_n}\rangle_d^{\qdot_0}.$ We give an equivalent characterization of when
 $\langle \sigma^P_{u_1}, \dots, \sigma^P_{u_n}\rangle_d^{\qdot_0}\neq 0$ in Theorem \ref{two}.
Now our first main theorem is the following (cf. Theorem \ref{main}):
\begin{theorem}\label{firstmain}
Let $(\mu_{1},\ldots,\mu_{n})\in \mathscr{A}^{n}$. Then, the following
are equivalent:

(a) $(\mu_{1},\ldots,\mu_{n})\in \mathscr{C}_{n}$,

(b) For any standard maximal parabolic
subgroup $P$ of $G$, any $u_{1},\ldots,u_{n}\in W^{P}$, and any $d\geq
0$ such that
$$\langle \sigma^{P}_{u_{1}},\dots,
\sigma^{P}_{u_{n}} \rangle_{d}^{\qdot_{0}}=1,
$$
the following inequality is satisfied:
$$
\mathscr{I}^{P}_{(u_{1},\ldots, u_{n};d)}:\medskip\medskip\medskip\,\,\,\,\,\,\,\,\,\,\sum^{n}_{k=1}
\omega_P(u_k^{-1}\mu_k) \leq d.
$$
\end{theorem}
The role of the flag varieties $(G/B)^n$ in \cite{Belkale-Kumar} is replaced here by the quasi-parabolic moduli stack $\Parbun_G$ of principal $G$-bundles on $\pone$ with parabolic structure at the marked points $b_1, \dots, b_n\in \pone$. The proof makes crucial use of the canonical reduction of parabolic $G$-bundles
and a certain {\it Levification process} of principal $P$-bundles (cf. Subsection \ref{section3.6}), which allows degeneration of a  principal $P$-bundle
to a $L$-bundle (a process familiar
in the theory of vector bundles as reducing the structure group to a Levi subgroup of $P$).

Our second main theorem (cf. Theorem \ref{irredundance}) asserts that the inequalities given by the (b)-part of the above theorem provide an irredundant system of inequalities defining the polytope $\mathscr{C}_{n}$.
 Specifically, we have the following result. This  result for $G=\SL_m$ was proved by Belkale combining the works \cite{BTIFR,BFulton} (see Remark \ref{FultonC}).  It is the multiplicative  analogue of Ressayre's result \cite{Re}. Our proof is a certain adaptation of Ressayre's proof (there are additional subtleties).
\begin{theorem} Let $n\geq 2$. The inequalities
$$\mathscr{I}^{P}_{(u_{1},\ldots, u_{n};d)}:\medskip\medskip\medskip\,\,\,\,\,\,\,\,\,\,\sum^{n}_{k=1}
\omega_P(u_k^{-1}\mu_k) \leq d,$$ given by  part (b) of the above theorem (as we run through the standard maximal parabolic subgroups $P$, $n$-tuples $(u_1, \dots, u_n)\in (W^P)^n$ and non-negative integers $d$ such that $
\langle \sigma^{P}_{u_{1}},\ldots, \sigma^{P}_{u_{n}} \rangle_{d}^{\qdot_{0}}=1)$  are pairwise distinct (even up to scalar multiples) and
form an irredundant system of inequalities
defining the eigen polytope $\mathscr{C}_{n}$ inside $\mathscr{A}^{n}$, i.e., the hyperplanes given by the equality in
$\mathscr{I}^{P}_{(u_{1},\ldots, u_{n};d)}$ are precisely the (codimension one) facets of the polytope $\mathscr{C}_{n}$ which intersect the interior of $\mathscr{A}^{n}$.
\end{theorem}
To show that the inequality
$\mathscr{I}^{P}_{(u_{1},\ldots, u_{n};d)}$
can not be dropped,
we produce (following Ressayre's general strategy \cite{Re}) a collection of points  of
$\mathscr{C}_{n}$ for which  the above inequality is  an equality, and such that their convex span has the dimension of
a facet (i.e., $-1+n\dim \frh$). This is achieved by the parabolic analogue of Narasimhan-Seshadri theorem for the Levi subgroup $L$ resulting
into a description of  $\mathscr{C}_{n}$ for $L$ in terms of the non-vanishing of the space of global sections of certain line bundles on
the moduli stack $\Parbun_L(d)$ of quasi-parabolic $L$-bundles of degree $d$ (cf. Theorem \ref{evproblem} applied to the semisimple part of $L$ and Corollary \ref{cor7.6}). To be able to use the parabolic analogue of Narasimhan-Seshadri theorem, we need
a certain {\it Levi twisting}, which produces an isomorphism of  $\Parbun_L(d)$ with  $\Parbun_L(d\pm 1)$ (cf. Lemma \ref{iso}).

Section 8 of the paper is greatly influenced by  Ressayre \cite{Re}  as will be clear to any informed reader.

It may be remarked that our work completes the multiplicative eigenvalue problem for compact simply-connected groups in the sense that we determine the multiplicative eigen polytope $\mathscr{C}_{n}$ by giving an irredundant  system of inequalities defining it. The problem of a recursive description of $\mathscr{C}_{n}$ in terms of eigen polytopes of
 ``smaller groups'' remains open for general $G$ (for $G=\operatorname{SL}(n)$ this has been carried out in \cite{b4}). 

\subsection{Acknowledgements.} We thank Anders Buch for providing the multiplicative table for the quantum cohomology of $G/P$, for $G$ of type $C_2$, and Patrick Brosnan for useful discussions. The authors were supported by the NSF grants DMS-0901249 and  DMS-1201310 respectively. We note that a result similar to Theorem \ref{firstmain} (and equivalent, by an argument in the manner of Theorem \ref{two} and Theorem 32 of \cite{Belkale-Kumar})  has
been established independently by Nicolas Ressayre \cite{r2}.
\section{ Notation and Preliminaries}\label{notation}

\subsection{Notation} Let $G$ be a  semisimple, connected, simply-connected  complex algebraic group. We choose a Borel subgroup
$B$ and a maximal torus $H\subset B$  and let $W=W_G:=N_G(H)/H$ be the associated Weyl group,
 where $N_G(H)$ is the normalizer of $H$ in $G$.  Let $P\supseteq B$ be a (standard)
 parabolic subgroup of $G$ and let $U=U_P$ be its unipotent radical. Consider the
 Levi subgroup $L=L_P$ of $P$ containing $H$, so that $P$ is the semi-direct product
 of $U$ and $L$. Then,  $B_L:=B\cap L$ is a Borel subgroup of
$L$. Let $\Lambda =\Lambda(H)$ denote the character group of $H$, i.e., the group of all the
algebraic group morphisms $H \to \Bbb G_m$.  Clearly, $W$ acts on $\Lambda$.
We denote the Lie algebras of $G,B,H,P,U,L,B_L$ by the corresponding Gothic characters:
$\fg,\fb,\fh,\fp,\fu,\fl,\fb_L$ respectively.  We will often identify an element
$\lambda$ of $\Lambda$
(via its derivative $\dot{\lambda}$) by an element of $\fh^*$. Let $R=R_\fg \subset \frh^*$ be
the set of roots of $\fg$ with respect to the Cartan subalgebra $\fh$ and let $R^+$ be the
 set of positive roots (i.e., the set of roots of $\fb$).
Similarly, let $R_\fl$ be the set of roots of $\fl$ with respect to $\fh$ and
$R_\fl^+$ be the set of roots of $\fb_L$. Let  $\Delta = \{\alpha_1, \dots, \alpha_\ell\}
\subset R^+$ be the set of simple
roots, $\{\alpha_1^\vee, \dots, \alpha_\ell^\vee\}
\subset \fh$ the
corresponding simple coroots and $\{s_1,\dots, s_\ell\}\subset W$ the
corresponding simple reflections,  where $\ell$ is the rank of $G$.
 We denote by $\Delta_P$ the set of simple roots
 contained in $R_\fl$ and we set
$$ S_P:=\Delta\setminus \Delta_P.$$  For any $ 1\leq
j\leq \ell$, define the element $x_j\in \fh$
by
\begin{equation}\label{eqn0}\alpha_i(x_{j})=\delta_{i,j},\text{ }\forall\text{ } 1\leq i\leq \ell.
\end{equation}
Further, define the element $\bar{x}_j$ by
\begin{equation}\label{eqn0'}\bar{x}_j=N_jx_j,
\end{equation}
where $N_j$ is the smallest positive integer such that $N_jx_j$ is in the coroot lattice $Q^\vee\subset \fh$ of $G$.

Recall that if $W_P$ is the Weyl group of $P$ (which is, by definition, the  Weyl
Group $W_L$ of $L$), then  each coset of $W/W_P$ contains  a unique member $w$ of minimal length.
 This satisfies:
\begin{equation}\label{eqn1}
wB_L w^{-1} \subseteq B.
\end{equation}
Let $W^P$ be the set of the minimal length representatives
in the cosets of $W/W_P$.

For any $w\in W^P$, define the Schubert cell:
\[
C_w^P:=  BwP/P \subset X^P:=G/P.
  \]
Then, it is a locally closed subvariety of the flag variety $X^P$, isomorphic with the affine
space $\Bbb A^{\ell(w)}, \ell(w)$ being the length of $w$. Its closure is denoted by $X^P_w$, which is an irreducible (projective) subvariety
of $X^P$ of dimension $\ell(w)$. We denote the point $wP\in C_w^P$ by $\dot{w}$.
We abbreviate $X_w^B$ by $X_w$. We define the shifted Schubert cell $\Lambda^P_w:= w^{-1}BwP/P$, and its closure is denoted  by
$\bar{\Lambda}^P_w$. Then, $B_L$ keeps $\Lambda^P_w$ (and hence $\bar{\Lambda}^P_w$) stable by \eqref{eqn1}.

Let
$\mu(X_w^P)$ denote the fundamental class of $X_w^P$
considered as an element of the singular homology with integral coefficients
$H_{2\ell(w)}(X^P, \Bbb Z)$ of $X^P$. Then, from the Bruhat decomposition, the elements
$\{\mu(X_w^P)\}_{w\in W^P}$ form a $\Bbb Z$-basis of
 $H_*(X^P, \Bbb Z)$. Let $\{\sigma^P_w\}_{w\in W^P}$ be the Poincar\'e dual basis of the singular
cohomology with integral coefficients $ H^*(X^P, \Bbb Z)$. Thus,
$\sigma^P_w\in H^{2(\dim X^P-\ell (w))}(X^P, \Bbb Z)$.

An element $\lambda\in \Lambda$ is called dominant (resp. dominant regular) if
$\dot{\lambda}(\alpha^\vee_i)\geq 0$ (resp. $\dot{\lambda}(\alpha^\vee_i) > 0$)
for all the simple coroots $\alpha^\vee_i$. Let $\Lambda_+$ (resp. $\Lambda_{++}$)
denote the set of all the dominant (resp. dominant regular) characters.
 We denote the fundamental weights by
$\{\omega_i\}_{1\leq i\leq \ell}$, i.e.,
$$\omega_i(\alpha_j^\vee)=\delta_{i,j}.$$

For any $\lambda \in \Lambda$, we have a  $G$-equivariant
 line bundle $\mathcal{L} (\lambda)$  on $G/B$ associated to the principal $B$-bundle $G\to G/B$
via the one dimensional $B$-module $\lambda^{-1}$. (Any  $\lambda \in \Lambda$ extends
 uniquely to a character of $B$.) The one dimensional $B$-module $\lambda$ is also denoted by
 $\Bbb C_\lambda$. If $\lambda$ vanishes on $\{\alpha_i^\vee\}_{\alpha_i\in \Delta_P}$, it defines a
character of $P$ and hence a line bundle $\mathcal{L}_P(\lambda)$ on $X^P$ associated to the character $\lambda^{-1}$ of $P$.
It is easy to see that
\begin{equation}\label{eqn1'}\int_{X_{s_i}^P}\,c_1(\mathcal{L}_P(\lambda))=\lambda(\alpha_i^\vee),\,\,\,\text{for any}\,\, \alpha_i\in S_P.
\end{equation}

For $w\in W^P$, define $\chi_w=\chi_w^P \in \frh^*$ by
$$\chi_w= \sum_{\beta\in (R^+\setminus R^+_\fl)\cap w^{-1} R^+} \beta= \rho-2\rho^L+ w^{-1}\rho,$$
where $\rho$ (resp. $\rho^L$) is half the sum of roots in $R^+$ (resp. in $R^+_\fl$).

All the schemes are considered over the base field of complex numbers $\Bbb C$. The varieties
are reduced (but not necessarily irreducible) schemes.

\subsection{Quantum cohomology of $X^P$}
We refer the reader to \cite{Kontsevich-M, Fulton-P} for the foundations of small quantum cohomology (also see \cite{Fulton-W}).
Let $X=X^P$ be the flag variety as above, where $P$ is any standard parabolic subgroup. Then,
$$\{\mu(X_{s_i}^P)\}_{\alpha_i\in S_P}$$
is a $\mathbb{Z}$-basis of $H_2(X, \bz)$.

Introduce variables $q_{i}$ associated to each $\alpha_i\in S_P$.
For
\begin{equation}\label{defD}
d=\sum_{\alpha_i\in S_P} a_{i} \mu(X_{s_i}^P) \in H_2(X, \bz),
\end{equation}
let $q^d:=\prod_{\alpha_i\in S_P} q_{i}^{a_{i}}$. We say $d\geq 0$ if each $a_{i}\geq 0$. We denote the class $d$ by $(a_i)_{\alpha_i\in S_P}.$

\begin{defi}\label{mansur}
Let $u_1,\dots,u_n\in W^P$ and $d\geq 0\in H_2(X)$. Fix distinct points $b_1,\dots,b_n\in \pone$, and a general point
$(g_1,\dots,g_n)\in G^n$. Let
\begin{equation}\label{paul}
\langle \sigma^P_{u_1},\sigma^P_{u_2},\dots,\sigma^P_{u_n}\rangle_d
\end{equation}
be the number of maps (count as $0$ if infinite)  $f:\pone\to X$ of degree $d$ (i.e., $f_*[\pone]=d\in H_2(X)$) such that $f(b_k)\in g_k C^P_{u_k}, k=1,\dots, n$.
\end{defi}
\begin{defi}\label{notnull}
Call a tuple $(P;u_1,\dots,u_n; d)$ as above {\it quantum non-null} if there are maps $f$ (possibly infinitely many) in the setting of Definition \ref{mansur}.
This notion will play a role in Section \ref{sufficiency}.
\end{defi}
The space of maps $\pone\to X$ of degree $d$ is  a smooth variety of dimension $\dim X +\int_d\,c_1 (T_X)$,
where $T_X$ is the tangent bundle of $X$. Therefore, \eqref{paul} is zero unless
\begin{equation}\label{normal}
\sum_{k=1}^n \codim  [\bar{\Lambda}^P_{u_k}] = \dim X + \int_d\,c_1 (T_X).
\end{equation}
Let $w_o$ (resp. $w_o^P$) be the longest element of the Weyl group $W$ (resp. $W_P$). Now, the quantum product in $H^*(X,\bz)\otimes_\bz\,\bz[q_i]_{\alpha_i\in S_P}$ is defined by
 \begin{equation}\label{eqqdef} \sigma^P_{u}\star\sigma^P_{v}=\sum_{d\geq 0} q^d\langle \sigma^P_u, \sigma^P_v,\sigma^P_w\rangle_d \sigma^P_{w_o w w_o^P},
\end{equation}
giving rise to a graded associative and commutative ring, where we assign the degree of $q_{i}$ to be $\int_{X_{s_i}^P} c_1(T_X)$, which is clearly equal to $2-2\rho^L(\alpha_i^\vee)$ by the equation \eqref{eqn1'}.

We note that there exist maps  $\pone\to X$ of any degree $d\geq 0$.

\section{Quantum Levi-movability and a deformed product in the quantum cohomology of $X^P$}\label{section3}

Consider a commutative and associative  ring $R$ over $\mathbb{Z}$ freely (additively) generated by $\{e_u\}_{u\in I}$. Write
$$e_u\cdot e_v=\sum c_{u,v}^w e_w,\,\,\,\text{for}\,\,  c_{u,v}^w\in \mathbb{Z}.$$
Consider a  multigrading $\gamma:I\to \Bbb{Z}^S$, where $S$ is a set with $m$ elements, such that whenever $c_{u,v}^w\neq 0$, we have
$$\gamma(w)-\gamma(u)-\gamma(v)\geq 0,$$
where  an element $\vec{a}=(a_i)_{i\in S}\in\Bbb{Z}^S$ is $\geq 0$ if each $a_i\geq 0$. Introduce $m=|S|$ variables $\tau_i,i\in S$. For $\vec{a}\in\Bbb{Z}^S$ define $\tau^{\vec{a}}=\prod_{i\in S}\tau_i^{a_i}$. Define a new product
$\odot_{\tau}$ on $R\tensor_{\Bbb{Z}}\Bbb{Z}[\tau_i]_{i\in S}$  by
$$e_u\odot_{\tau} e_v=\sum \tau^{\gamma(w)-\gamma(u)-\gamma(v)} c_{u,v}^w e_w.$$
\begin{lemma}\label{lem3.1}
\begin{enumerate}
\item $\odot_{\tau}$ is a commutative and associative ring.
\item Putting all $\tau_i=0$ gives a commutative and associative graded ring (i.e., the product respects the grading). More precisely, the ring structure $\odot_0$ is given by the following:
$$e_u\odot_{0} e_v=\sum' c_{u,v}^w e_w,$$
where the sum is restricted over $w$ such that $\gamma(w)=\gamma(u)+\gamma(v)$.
\end{enumerate}
\end{lemma}
\begin{example} The deformed product in $H^*(X^P)$ as introduced by Belkale-Kumar in \cite{Belkale-Kumar} comes from such a situation with
$$\gamma(u)=(\chi_u(x_i))_{\alpha_i\in S_P},$$
and the Schubert basis $\{\sigma^P_w\}_{w\in W^P}.$
\end{example}

Let $X=X^P$ be any flag variety. Recall the definition of the small quantum cohomology of $X$ from Section 2. We give the definition of a certain deformation of the quantum product in $X$ as below.

We begin with the following result, which will be proved towards the end of this section.
\begin{theorem}\label{one} Let $u_1, \dots, u_n\in W^P$ and $d=(a_i)_{\alpha_i\in S_P} \in H_2(X,\bz)$ be such that
 $\langle \sigma^P_{u_1},\sigma^P_{u_2},\dots,\sigma^P_{u_n}\rangle_d\neq 0$. Then, for any ${\alpha_i\in S_P}$,
$$ (\chi_{e}-\sum_{k=1}^n\chi_{u_k})(x_i) + \sum _{\alpha\in R^+\setminus R^+_{\fl}} \alpha(x_i)\alpha(\tilde{d})\geq 0,$$
where $\tilde{d}=\sum_{\alpha_j\in S_P} a_{j}\alpha_j^\vee.$

\end{theorem}

Consider the normalized Killing form  $\langle\,,\rangle$ on $\fh^*$ normalized so that $\langle\theta,\theta\rangle =2$, where
$\theta$ is the highest root of $\fg$. This gives rise to an identification
$$\kappa:\fh^*\to \fh. $$
It is easy to see that
\begin{equation}\label{eq20} \kappa(\omega_i)= \frac{\langle\alpha_i,\alpha_i\rangle}{2}x_i.
\end{equation}
\begin{lemma}\label{coxeter} For any $h,h'\in \frh$,
$$ \sum_{\alpha\in R}\alpha(h)\alpha(h')=2g^*\langle h, h'\rangle,$$
where $g^*:=1+\langle \rho, \theta^\vee\rangle$ is the dual Coxeter number of $\frg$.
\end{lemma}
\begin{proof}
Consider the bilinear form on $\frh$ given by $\langle h, h'\rangle':= \sum_{\alpha\in R}\alpha(h)\alpha(h')$. It  is $W$-invariant and hence it is a multiple of the original Killing form, i.e., $\langle h, h'\rangle'= z\langle h, h'\rangle$, for some constant $z$. To calculate
$z$,
\begin{align*} 2z&=\langle \theta^\vee, \theta^\vee\rangle'\\
&=\sum_{\alpha\in R}\alpha(\theta^\vee)^2\\
&=4+2\sum_{\alpha\in R^+}\alpha(\theta^\vee), \,\,\,\text{since}\,\, \alpha(\theta^\vee)\in \{0,1\} \,\,\forall \,\alpha\in R^+\setminus\{\theta\}\\
&=4+4\rho(\theta^\vee)\\
&=4g^*.
\end{align*}
\end{proof}

Theorem \ref{one} and the general deformation principle spelled out in Lemma \ref{lem3.1}  allows us to give the following deformed product in the quantum cohomology of $X$.

\begin{defi}  \label{qdeformquantum} Introduce the $\tau$-deformation of the quantum cohomology of $X$ as follows:

As a $\bz[q,\tau]$-module, it is the same as $H^*(X,\bz)\otimes_\bz \bz[q,\tau]$, where $q$ (resp. $\tau$) stands for
multi variables $\{q_i\}_ {\alpha_i\in S_P}$ (resp. $\{\tau_i\}_ {\alpha_i\in S_P}$). For $u,v\in W^P$, define the $\bz[q,\tau]$-linear
{\it deformed product} by
$$\sigma^P_{u}\qdot\sigma^P_{v}=\sum_{d\geq 0\in H_2(X,\bz); w\in W^P} \Bigl(\prod_{\alpha_i \in S_P} \tau_i^{A_i(u,v,w,d)}\Bigr)q^d\langle \sigma^P_u, \sigma^P_v,\sigma^P_w\rangle_d \sigma^P_{w_o w w_o^P},$$
where (for $d=(a_i)_{\alpha_i\in S_P}$)
\begin{equation} A_i(u,v,w,d) =(\chi_e-\chi_u-\chi_v-\chi_w)(x_i)+\frac{2 a_ig^*}{\langle\alpha_i,\alpha_i\rangle}.
\end{equation}
Using  Lemma \ref{coxeter} and the equation \eqref{eq20} (and observing that for $\alpha\in  R^+_{\fl},  \alpha(x_i)=0$
for any $\alpha_i\in S_P$), we get another expression:
$$A_i(u,v,w,d) =(\chi_e-\chi_u-\chi_v-\chi_w)(x_i)+ \sum _{\alpha\in R^+\setminus R^+_{\fl}} \alpha(x_i)\alpha(\tilde{d}).$$

Evaluating each $\tau_i=0$ in the above (which is well defined because of Theorem \ref{one}), we get
$$\sigma^P_{u}\qdot_0\sigma^P_{v}=\sum'_{d, w} q^d\langle \sigma^P_u, \sigma^P_v,\sigma^P_w\rangle_d \sigma^P_{w_o w w_o^P},$$
where the sum is restricted over those $d\geq 0 \in H_2(X,\bz)$ and $w\in W^P$  so that $A_i(u,v,w,d)=0$ for all $\alpha_i\in S_P$.
We shall denote the coefficient of $q^d \sigma^P_{w_o w w_o^P}$ in  $\sigma^P_{u}\qdot_0\sigma^P_{v}$ by $\langle \sigma^P_u, \sigma^P_v,\sigma^P_w\rangle_d^{\qdot_0}.$

From the general deformation principle given  in Lemma \ref{lem3.1}, taking the multigraded function $\gamma=(\gamma_i)_{\alpha_i\in S_P}$ defined by
$$\gamma_i(q^d\sigma_w^P)=\chi_w(x_i)+\frac{2a_ig^*}{\langle \alpha_i, \alpha_i\rangle}\,\,\,\text{for}\,\, d=(a_j)_{\alpha_j\in S_P}\in H_2(X, \bz),$$
it follows that
$\qdot$ and $\qdot_0$ give associative (and commutative) products.
\end{defi}
\begin{lemma}\label{minuscule}
Let $P$ be a cominuscule maximal standard parabolic
 subgroup of $G$ (i.e., the unique simple root $\alpha_{i_P} \in S_P$
 appears with coefficient
$1$ in the highest root of $R^+$). Then, the deformed product
$\qdot$ coincides with the quantum  product $\star$ in $H^*(X^P)\otimes_\Bbb{Z}\,\Bbb{Z}[q_{i_P}]$.
\end{lemma}
\begin{proof} By the definition of $\qdot$, it suffices to show that for any $u,v, w\in W^P$ and $d=a_{i_P}\in H_2(X^P, \Bbb{Z})$, such that
$\langle \sigma^P_u, \sigma^P_v,\sigma^P_w\rangle_d \neq 0$,
\begin{equation} \label{eqn5.1}
 A_{i_P}(u,v,w,d)= 0.
\end{equation}
Since $P$ is cominuscule, by \cite{Belkale-Kumar}, Proof of Lemma 19,
\begin{equation} \label{eqn5.2}
 \chi_w (x_{i_P})
=\codim(\Lambda^{P}_w:X^P).
\end{equation}
Moreover, the quantum cohomological degree of $q_{i_P}$ equals
\begin{equation} \label{eqn5.1'} \int_{X^P_{s_{i_P}}} c_1(T_{X^P})=\sum _{\alpha\in R^+\setminus R^+_{\fl}} \alpha(\alpha_{i_P}^\vee).
\end{equation}
 Thus, since $\langle \sigma^P_u, \sigma^P_v,\sigma^P_w\rangle_d \neq 0$, we get (by equating the cohomological degrees on the two sides
of \eqref{eqqdef})
\begin{equation} \label{eqn5.3}
\codim(\Lambda^{P}_u:X^P)+\codim(\Lambda^{P}_v:X^P)
=\dim \Lambda^{P}_w + a_{i_P}\cdot \text{degree} \,q_{i_P}.
\end{equation}
Combining equations \eqref{eqn5.2}, \eqref{eqn5.1'}  and \eqref{eqn5.3}, we get equation \eqref{eqn5.1} (since $\alpha(x_{i_P})=1$ for all
$\alpha\in R^+\setminus R^+_{\fl}$, $P$ being cominuscule).
\end{proof}

\subsection{The enumerative problem of small quantum cohomology in terms of principal bundles}\label{subsection3.1}
Let $\mathcal{E}$ be a principal right $G$-bundle on $\pone$. It is standard that sections $f:\pone\to \mathcal{E}/P$ are in one to one correspondence with reductions of the structure group of
$\mathcal{E}$ to $P$. This correspondence works as follows: Given $f$, let
$\mathcal{P}$ be the right $P$-bundle with fiber $f(x)P\subseteq \mathcal{E}_x$ over $x\in \pone$. It is then easy to see
that there is a canonical isomorphism of principal $G$-bundles $\mathcal{P}\times^P G\to \mathcal{E}$.
For $\mathcal{E}=\epsilon_G$, the trivial bundle $\pone\times G\to \pone$,
\begin{enumerate}
\item Sections $f$ correspond to maps $\bar{f}:\pone\to X^P$.
\item For $\alpha_i\in S_P$, let $\mathcal{E} ({\omega_i}):=\epsilon\boxtimes \mathcal{L}_P(\omega_i)$ be the corresponding line bundle on $\mathcal{E}/P
=\pone\times X^P$, where $\epsilon$ is the trivial line bundle on $\pone$. Then, $\bar{f}$ has degree $d=
(a_i)_{\alpha_i \in S_P}$
if $c_1(f^*(\mathcal{E}(\omega_i)))=a_{i}$ (using the identity \eqref{eqn1'}).
\end{enumerate}

Let $\mathcal{E}$ be a principal $G$-bundle on $\pone$. We want to state an enumerative problem for $\mathcal{E}$ that corresponds to
that of Definition \ref{mansur} for $\mathcal{E}=\epsilon_G$. Fix distinct points $b_1,\dots,b_n\in \pone$,    $u_1,\dots,u_n\in W^P$ and
 $d=(a_i)_{\alpha_i \in S_P} \geq 0\in H_2(X^P)$. Fix general choices of $\bar{g}_k\in \mathcal{E}_{b_k}/B$. The enumerative problem,
which gives the Gromov-Witten numbers  $\langle \sigma^P_{u_1},  \dots,\sigma^P_{u_n}\rangle_d$ in the case
$\mathcal{E}=\epsilon_G$, is the following:
Count the number of sections $f:\pone\to \mathcal{E}/P$ (count as $0$ if infinite), such that
\begin{enumerate}
\item $c_1(f^*(\mathcal{E}(\omega_i)))=a_{i}$ for each $\alpha_i\in S_P$, where $\mathcal{E}(\omega_i)$ is the line bundle $\mathcal{E}\times^P \Bbb{C}_{-\omega_i}$ on $\mathcal{E}/P$.
\item $f(b_k)\in \mathcal{E}_{b_k}/P$ and $\bar{g}_k\in\mathcal{E}_{b_k}/B$ are in {\it relative position} $u_k\in W^P, k=1,\dots,n$,  defined as follows. Pick a trivialization
$e\in \mathcal{E}_{b_k}$ and write
$f(b_k)=e h_k P$ and $\bar{g}_k= e g_kB$. Then, we want $h_k\in g_k B u_k P\subseteq X^P$. A different choice of  $e$ acts on $h_k$ and $g_k$ by
a left multiplication and therefore does not affect the relative position.
\end{enumerate}
The above enumerative problem may be degenerate for some $\mathcal{E}$.
\subsection{Tangent spaces}\label{formules}

 Since $X=X^P$ is a homogeneous space, the tangent bundle $T_X$ is globally generated (since $\frg\tensor\mathcal{O}_X$ surjects onto $T_X$).

Fix any $\alpha_i\in S_P$. We can filter $T_{\dot{e}}:=T(X^P)_{\dot{e}}\simeq \fg/\fp$ by counting the multiplicity of  $\alpha_i$ in the root spaces
in $\fg/\fp$, where $\dot{e}$ is the base point of $X$. Specifically, for any $r\geq 1$,
let  $T_{i,r}\subset T_{\dot{e}}$ be the P-submodule spanned by the root spaces  $\frg_{-\alpha}$ of $T_{\dot{e}}$ such that $\alpha(x_i)\leq r$.
Define the $P$-module $ Q_{i,r}$ by the following:
$$0\to T_{i,r}\to  T_{\dot{e}}\to Q_{i,r}\to 0.$$

Let $\mq_{i,r}$ (resp. $\mathcal{T}_{i,r}$) be the vector bundle on $X$ arising from the $P$-module $Q_{i,r}$ (resp. ${T}_{i,r}$).
Since $\mq_{i,r}$ are quotients of $T_X$, they are globally generated. Let $\beta_i:=\sum_{r\geq 1}c_1(\mq_{i,r})$,  and define the integers $s_{i,j}=\int_{X^P_{s_j}}\beta_i$, for any $\alpha_i, \alpha_j\in S_P$. Then, it is easy to see that

\begin{equation}\label{eqn2'} s_{i,j}=\sum _{\alpha\in R^+\setminus R^+_{\fl}}\alpha(x_i)\alpha(\alpha^\vee_j).
\end{equation}
This is a non-negative integer since $\mq_{i,r}$'s are globally generated.

\subsection{Some deformation theory}
Let $\mathcal{E}$ be a principal $G$-bundle on a smooth curve $C$ and let $P\subseteq G$ be  a parabolic subgroup. Let $f:C\to \mathcal{E}/P$  be a section,
and $\mathcal{P}$ the corresponding $P$-bundle.

Let $Z$ be the space of sections $f:C\to \mathcal{E}/P$. This is a subscheme of the scheme $\mathcal{Z}$ of maps  $\beta: C\to \mathcal{E}/P$.
Let $M$ be the scheme of maps $C \to C$. Then, we have the morphism $\phi: \mathcal{Z}\to M, \beta \mapsto \gamma\circ \beta$, where
$\gamma:\mathcal{E}/P\to C$ is the canonical projection. By definition, $Z$ is the fiber of $\phi$ over the identity map $I_C$.

\begin{lemma}
 For any $f\in Z$, the Zariski tangent space $TZ_f$ is identified with $H^0(C, f^* T_v(\mathcal{E}/P))$, where $T_v$ is the vertical tangent bundle.
\end{lemma}
\begin{proof}
By deformation theory, the tangent space of $\mathcal{Z}$ at the point $f$ is
$H^0(C,f^*T(\mathcal{E}/P))$. There
is a natural exact sequence:
$$0\to f^* T_v(\mathcal{E}/P)\to  f^* T(\mathcal{E}/P)\to TC\to 0,$$
which allows us to conclude the proof.
\end{proof}

 Let $f\in Z$ (for $C = \pone$) and let $\mathcal{P}$ be the corresponding principal $P$-bundle. Then, $Z$ is smooth at $f$ of the expected dimension
$$ \dim \,X+\int_{\pone} f^*(c_1(T_v(\mathcal{E}/P)))$$ if
 $H^1(\pone,f^*T_v(\mathcal{E}/P))=0$ (which happens if, for example,  $\mathcal{E}$ is trivial cf. \cite{kollar}). If $Z$ is smooth of the expected dimension at $\mathcal{P}$ then $f$ deforms with every deformation of $\mathcal{E}$  (over a complete local ring).

We have the following simple result.
\begin{lemma}
$f^*T_v(\mathcal{E}/P)=\mathcal{P}\times^P T_{\dot{e}}$.
\end{lemma}

Also, note that for any character $\beta$ of $P$,
\begin{equation}\label{eqn2"} f^* \mathcal{E} (\beta)= \mathcal{P}\times^P \Bbb{C}_{-\beta}\,\,\,\text{ as a bundle over}\,\,  C.
\end{equation}

\subsection{Tangent spaces of Schubert varieties}
Let $\mathcal{P}$ be a principal $P$-bundle on $\pone$ and $x\in\pone$. Given $\bar{p}\in \mathcal{P}_x/B_L$ and $u\in W^P$, we can construct a subspace
$T(\bar{p},u,x)\subseteq \mathcal{P}_x\times^P T_{\dot{e}}$ as follows. Fix a  trivialization $e$ of $\mathcal{P}_x$ and write  $\bar{p}=epB_L$.
Then, the subspace $T(\bar{p},u,x)$ is defined to be $e\times T(p\Lambda^P_u)_{\dot{e}}\subseteq \mathcal{P}_x\times^P T_{\dot{e}}$.
A different choice of the coset representative of $\bar{p}$ or the choice of $e$ gives the same subspace.

Consider the evaluation map $e_{b_k}: Z\to \mathcal{E}_{b_k}/P$ at $b_k$. Fix  $f\in Z$ such that the corresponding principal bundle is $\mathcal{P}$. Then, the
differential map  $de_{b_k}$ on tangent spaces
$$H^0(\pone, \mathcal{P}\times^P T_{\dot{e}})\to T_{f(b_k)}(\mathcal{E}_{b_k}/P)=\mathcal{P}_{b_k}\times^P T_{\dot{e}}$$
 is the evaluation map at $b_k$.

Fix an element $e_k\in \mathcal{P}_{b_k}.$ Since we require that $f(b_k)$ and $\bar{g}_k$ are in relative position $u_k\in W^P$, we get that
$\bar{g}_k=e_kp_ku_k^{-1}B$, for some $p_k\in P$. To prove this, observe that since $f(b_k)=e_kP$ and $\bar{g}_k=e_kg_kB \in \mathcal{E}_{b_k}/B$
 are in relative position $u_k$, we get $1\in g_kBu_kP$, i.e., $g_kc_ku_kp_k^{-1}=1$ for some $c_k\in B$ and $p_k\in P$. From this we see that
$\bar{g}_k=e_kp_ku_k^{-1}B$.

 \subsection{Determinant of cohomology}\label{detnew}
The determinant of cohomology of a coherent sheaf $\mf$ on a curve $C$ is the line $$D(\mf)=\det H^0(C,\mf)^*\tensor \det H^1(C,\mf).$$
Automorphisms of $\mathcal{F}$ act on $D(\mf)$. For example, multiplication by $t\neq 0$ on $\mathcal{F}$ acts on $D(\mf)$ by $t^{-\chi(C,\mf)}$,
where $\chi(C,\mf)$ is the Euler characteristic of $\mf$. In the cases we consider here, $\mf$ is locally free.

Suppose $\chi(C,\mf)=0$, then,  as,  e.g., in \cite{Fa},  $D(\mf)$ carries a canonical element $\theta(\mf)$ which is non-vanishing if and only if
\begin{equation} \label{eqnew1} H^0(C,\mf)=H^1(C,\mf)=0.
\end{equation}
Automorphisms of $\mathcal{F}$ act on $D(\mf)$ preserving $\theta(\mf)$. In particular, if an automorphism
of $\mf$ acts non-trivially on $D(\mf)$, then $\theta(\mf)=0$.

As earlier, we have fixed distinct points $b_1,\dots,b_n\in\pone$. Let $\Parbun_P=\Parbun_P(d)$ be the moduli-stack of quasi-parabolic principal $P$-bundles on $\pone$ of degree $d=(a_i)_{\alpha_i\in S_P}$,  i.e., data $\tilde{\mathcal{P}}=(\mathcal{P};\bar{p}_1,\dots,\bar{p}_n)$
such that
\begin{itemize}
\item $\mathcal{P}$ is a principal $P$-bundle on $\pone$ such that $\mathcal{P}\times^P \Bbb{C}_{-\omega_i}$ has degree $a_{i}$,
 for each $\alpha_i\in S_P$.
\item For $k=1,\dots,n$, $\bar{p}_k\in \mathcal{P}_{b_k}/B_L$.
\end{itemize}

\subsection{Transversality}\label{trans}
For any $\tilde{\Pc}=(\mathcal{P};\bar{p}_1,\dots,\bar{p}_n) \in \Parbun_P(d)$ and $u_1, \dots, u_n\in W^P$ satisfying the equation \eqref{normal}, define a locally free sheaf $\mathcal{K}=\mathcal{K} ({\tilde{\Pc}})$ on $\pone$ by the exact sequence coming from the evaluation map:
\begin{equation}\label{exacto}
0\to \mathcal{K}\to \mathcal{P}\times^P T_{\dot{e}}\to\bigoplus_{k=1}^n  i_{b_k*}\frac{\mathcal{P}_{b_k}\times^P T_{\dot{e}}}{T(\bar{p}_k, u_k, b_k)} \to 0,
\end{equation}
where $i_{b_k}$ is the embedding $b_k \hookrightarrow \pone$. Note that the Euler characteristic of $\mathcal{P}\times^P T_{\dot{e}}$ equals  $\dim X+\int_dc_1(T_X)$, which is the same as the  Euler characteristic of $\bigoplus_{k=1}^n  i_{b_k*}\frac{\mathcal{P}_{b_k}\times^P T_{\dot{e}}}{T(\bar{p}_k, u_k, b_k)}$  by the condition \eqref{normal}. Hence, $\mathcal{K}$ has zero Euler characteristic.
The transversality condition on $\tilde{\mathcal{P}}$ is the requirement that $\mathcal{K}$ has non-vanishing $\theta$-section $\theta(\mathcal{K})\in D(\mathcal{K})$. Note that this sort of reformulation of transversality appears in \cite{PB,b4,Belkale-Kumar}, and (in a related situation) in \cite{Belkale1}.

Define  a line bundle $\mar$ on $\Parbun_P$ such that its fiber over a quasi-parabolic $\tilde{\mathcal{P}}$ is $D(\mk(\tilde{\mathcal{P}}))$.
The line bundle $\mar$ admits a canonical section $\theta$ given by
 $$\theta(\tilde{\mathcal{P}}) = \theta(\mathcal{K}(\tilde{\mathcal{P}})).$$
\begin{remark}\label{rhone}
Note that if $\theta$ does not vanish at a quasi-parabolic $\tilde{\mathcal{P}}$ then, since
$$H^0(\pone, \mathcal{K})= H^1(\pone, \mathcal{K}) = 0 \,\,\,\text{by the identity} \,\eqref{eqnew1},$$
we get
$H^1(\pone,\mathcal{P}\times^P T_{\dot{e}})=0$ from the long exact cohomology sequence associated to the sheaf exact sequence \eqref{exacto}.
\end{remark}
\subsection{The space of $P$-subbundles of a $G$-bundle} \label{space} Let $\Parbun_G$ be the moduli-stack of quasi-parabolic principal $G$-bundles on $\pone$,  i.e., data $\tilde{\mathcal{E}}=(\mathcal{E};\bar{g}_1,\dots,\bar{g}_n)$, where $\E$ is a principal $G$-bundle on
$\pone$ and  $\bar{g}_k\in \mathcal{E}_{b_k}/B$.

For any $\tilde{\mathcal{E}}\in \Parbun_G$, a standard parabolic $P$, elements $u_1, \dots, u_n \in W^P$  and degree $d=(a_i)_{\alpha_i\in S_P}$, define the schemes $Z_d(\mathcal{E})$ and $Z_d'(\tilde{\mathcal{E}})$ as follows:
$$ Z_d(\mathcal{E}) \,\,\,\text{is the space of sections of }\,\, \mathcal{E}/P\,\,\,\text{of degree}\,\, d.$$
$$Z_d'(\tilde{\mathcal{E}})=Z_d'(\tilde{\mathcal{E}}; u_1, \dots, u_n):=\{f\in Z_d(\mathcal{E}): f(b_k) \,\,\,\text{and}\,\,\bar{g}_k\,\,\,\text{are in relative position}\,\,\,u_k \,\,\forall
1\leq k\leq n\}.$$
Then, for any $f\in Z_d(\mathcal{E})$, the Zariski tangent space
$$T(Z_d(\mathcal{E}))_f=H^0(\pone, \mathcal{P}(f)\times^P T_{\dot{e}}),$$
where $\mathcal{P}(f)$ is the $P$-subbundle of $\mathcal{E}$ associated to $f$. Thus, $f\in Z_d(\mathcal{E})$ is a smooth point if
$H^1(\pone, \mathcal{P}(f)\times^P T_{\dot{e}})=0.$

For any $f\in Z_d'(\tilde{\mathcal{E}})$, we have the canonical morphism
$$\phi_f:  H^0(\pone,\mathcal{P}(f)\times^P T_{\dot{e}})\to \bigoplus_{k=1}^n  \frac{\mathcal{P}(f)_{b_k}\times^P T_{\dot{e}}}{T(\bar{p}_k, u_k, b_k)},$$
induced from  the evaluation maps $e_{b_k}:Z_d(\E)\to \E_{b_k}/P$ at
$b_k$, where $\bar{p}_k$ is  any element   of $\mathcal{P}(f)_{b_k}/B_L$ such that  $\bar{g}_kBu_kP= \bar{p}_k\Lambda_{u_k}$.
(Observe that $\bar{p}_k$ is unique modulo the stabilizer of $\Lambda_{u_k}$ in $P$.)
Then, for any $f\in Z_d'(\tilde{\mathcal{E}})$, the Zariski tangent space
\begin{equation}\label{eqn3.10}T(Z'_d(\tilde{\mathcal{E}}))_f = \ker \phi_f.
\end{equation}
\begin{lemma}\label{wash} Let $u_1, \dots, u_n\in W^P$ and non-negative $d=(a_i)_{\alpha_i\in S_P}\in H_2(X, \bz)$.
 Then, the  following are equivalent under the condition \eqref{normal}.

(a)  $\langle \sigma^P_{u_1},\sigma^P_{u_2},\dots,\sigma^P_{u_n}\rangle_d\neq 0$,

(b)  There is a quasi-parabolic $P$-bundle $\tilde{\mathcal{P}}=(\mathcal{P};\bar{p}_1,\dots,\bar{p}_n) \in \Parbun_P(d)$  on $\pone$ so that
the canonical evaluation map
$$H^0(\pone, \mathcal{P}\times^P T_{\dot{e}})\to\bigoplus_{k=1}^n  \frac{\mathcal{P}_{b_k}\times^P T_{\dot{e}}}{T(\bar{p}_k, u_k, b_k)}$$
is an isomorphism.

(c) The section $\theta \in H^0(\Parbun_P(d),\mar)$ is nonzero.
\end{lemma}
\begin{remark} \label{theta} Observe that the section $\theta$ on a quasi-parabolic $\tilde{\mathcal{P}}$ does not vanish if and only if the evaluation map as in (b) of the above lemma is an isomorphism. In particular, in this case,   $H^1(\pone,\mathcal{P}\times^P T_{\dot{e}})=0$. (To prove this, use the identity \eqref{eqnew1} and the fact that $\chi(\pone, \mathcal{K}) = 0$.)
\end{remark}
\begin{proof} (of Lemma \ref{wash})
We first prove $(a) \Longrightarrow (b)$:
Let $Z_d$ be the space $Z_d(\epsilon_G)$,  where $\epsilon_G$ is the trivial $G$-bundle on $\pone$,
 i.e.,
$Z_d$ is the space of all maps $f:\pone \to X$  of degree $d$. Then, $Z_d$
is a smooth variety of dimension $=\dim X+\int_d c_1(T_X)$. Let $\{\bar{g}_k\}_{1\leq k\leq n}$ be general points of $G/B$. Then, by the assumption (a),
there exists $f\in Z'_d=Z'_d(\tilde{\epsilon}_G; u_1, \dots,  u_n)$, where  $\tilde{\epsilon}_G=(\epsilon_G; \bar{g}_1, \dots, \bar{g}_n)$ is the quasi-parabolic $G$-bundle on $\pone$.
Moreover, the subscheme $Z_d'$  is finite and reduced.  Fix $f\in Z_d'$. Then,  $Z_d'$ being finite and reduced, we get that $\phi_f$ is injective
by the equation \eqref{eqn3.10}. Now, by the condition \eqref{normal}, the dimension of the domain is at least as much as the dimension of  the range of $\phi_f$
(since $\chi(\pone, \mathcal{P}(f)\times^P T_{\dot{e}})= \dim X+\int_{\pone}c_1(T_X)$).  Hence, being injective,  $\phi_f$ is an isomorphism, proving (b).

$(b) \Longrightarrow (c)$: If the condition (b)  holds for $\tilde{\mathcal{P}}$ , then $H^0(\pone, \mathcal{K}(\tilde{\mathcal{P}}))
= 0.$ But, since $\chi(\pone, \mathcal{K}(\tilde{\mathcal{P}}))
= 0$, we get that $H^1(\pone, \mathcal{K}(\tilde{\mathcal{P}}))
= 0.$ Hence,  as in $\S$\ref{detnew},  $\theta(\mk(\tilde{\mathcal{P}}))\neq 0$, proving (c).

$(c) \Longrightarrow (a)$:
Suppose $\tilde{\mathcal{P}}=(\mathcal{P};\bar{p}_1,\dots,\bar{p}_n)$ is in $\Parbun_P(d)$ such that $\theta(\mk(\tilde{\mathcal{P}}))\neq 0$.
Let $\mathcal{E}=\mathcal{P}\times^P G$ be the corresponding principal $G$-bundle. The reduction of  $\mathcal{E}$ to $\mathcal{P}$
gives rise to the section $f:\pone\to \mathcal{E}/P$ of degree $d$. Consider the elements:
$$\bar{g}_k= \bar{p}_k u_k^{-1} B\in \mathcal{E}_{b_k}/B.$$
(Observe that $\bar{g}_k$ does not depend upon the choice of $\bar{p}_k$ in its $B_L$-orbit by the identity \eqref{eqn1}.)
Then,  $f(b_k)$ and $\bar{g}_k\in \mathcal{E}_{b_k}/B$
are in relative position $u_k$.  To see this, let $e_k$ be a trivialization of $\mathcal{P}_{b_k}$ and write $\bar{p}_k=e_kp_k$.
Then, $f(b_k)=e_kP $ and $\bar{g}_k=e_kp_k u_k^{-1}B$.

Thus, $f\in Z'_d=Z_d'(\tilde{\mathcal{E}})$, where $\tilde{\mathcal{E}}=(\mathcal{E}; \bar{g}_1, \dots, \bar{g}_k)$.  By Subsection \ref{space}, the Zariski tangent space to $Z_d'$ at $f$ is equal to $H^0(\pone, \mathcal{K}(\tilde{\mathcal{P}}))$,
which is zero by assumption. By Remark \ref{theta},  $H^1(\mathcal{P}\times^P T_{\dot{e}})=0$.

Consider a one parameter family of deformations  $\tilde{\mathcal{E}}_t=(\mathcal{E}_t; \bar{g}_1(t), \dots, \bar{g}_n(t))$ parameterized by a smooth curve $S$ so that at a marked point $0\in S$,
 $\tilde{\mathcal{E}}_0= \tilde{\mathcal{E}}$ and the underlying bundle ${\mathcal{E}}_t$ is trivial for general $t\in S$. We then have families
$\pi: \mathcal{Z}_d \to S$ and $\pi':\mathcal{Z}'_d \to S$ over $S$ with fiber $Z_d(\mathcal{E}_t)$ and $Z'_d(\tilde{\mathcal{E}}_t)$ respectively. Thus, $f\in \mathcal{Z}'_{0}$. We claim that $\pi'$ is a dominant morphism  of relative dimension zero at $f$:

Observe first  that $\pi$ is smooth at $f$ using the fact (noted above) that $H^1(\mathcal{P}\times^P T_{\dot{e}})=0$. 
Since $\theta(\mk(\tilde{\mathcal{P}}))\neq 0$, there exists a neighborhood of $\tilde{\mathcal{P}}$ in $\mathcal{Z}_d $ such that for any
 $\tilde{\mathcal{Q}}$ in the neighborhood, $\theta(\mk(\tilde{\mathcal{Q}}))\neq 0$. Moreover, restricted to this neighborhood, $\pi$ is a smooth morphism and $\pi'$ has finite fibers. Choose a lift $g_k$ of $\bar{g}_k$, i.e., a section $g_k:S \to \cup_t \E_t$ such that $g_k(t)\in \E_t$ is a
lift of $\bar{g}_k(t)\in \E_t/B$. This is, of course, possible replacing $S$ (if needed) by a smaller \'etale neighborhood of $0\in S$. Thus, there is a
neighborhood $\mathcal{Z}_d^o$ of $f$ in $\mathcal{Z}_d$ and a morphism
$$\beta: \mathcal{Z}_d^o\to (X^P)^n,\,\,\,\,\,\beta(\tilde{\mathcal{Q}})=(h_kP)_{1\leq k\leq n},$$
where $h_kP$ is the unique element such that $g_k(\pi(\tilde{\mathcal{Q}}))\in \tilde{\mathcal{Q}}_{b_k}h_k^{-1}$. Moreover, we can choose $\mathcal{Z}_d^o$ small enough so that $\pi_{|\mathcal{Z}_d^o}: \mathcal{Z}_d^o \to S$ is a smooth morphism and
$\pi'_{|\mathcal{Z}_d'\cap \mathcal{Z}_d^o}: \mathcal{Z}_d'\cap \mathcal{Z}_d^o \to S$ has finite fibers. From the definition of
$\mathcal{Z}_d'$, it is clear that $$\mathcal{Z}_d'\cap \mathcal{Z}_d^o=\beta^{-1}(C_{u_1}^P\times \dots \times C_{u_n}^P).$$
Since $X^P$ is smooth, $C_{u_1}^P\times \dots \times C_{u_n}^P$ is locally defined by exactly $r$ equations, where $r$ is the codimension of
$C_{u_1}^P\times \dots \times C_{u_n}^P$ in $(X^P)^n$. Hence, by \cite[Exercise 3.22, Chap. II]{hartshorne},
$$\dim_f(\mathcal{Z}_d')\geq \dim_f(\mathcal{Z}_d^o)-\sum_{k=1}^n \codim C^P_{u_k}=1,$$
where the last equality follows since $\phi_f$ (defined in Subsection \ref{space}) is an isomorphism and $\pi$ is a smooth morphism. But, since
$\pi'_{|\mathcal{Z}_d'\cap \mathcal{Z}_d^o}$ has only finite fibers, $\pi'_{|\mathcal{Z}_d'\cap \mathcal{Z}_d^o}: \mathcal{Z}_d'\cap \mathcal{Z}_d^o \to S$  is a dominant morphism. This proves (a).
\end{proof}
\begin{remark} Even though we do not need, the map $\pi':\mathcal{Z}'_d \to S$ is, in fact, a flat morphism in a neighborhood of $f$ (with fiber dimension $0$).
\end{remark}

Similarly, let $\Parbun_L=\Parbun_L(d)$ be the moduli-stack of quasi-parabolic principal $L$-bundles on $\pone$ of degree $d$, i.e., data $\tilde{\mathcal{L}}= (\mathcal{L};\bar{l}_1,\dots,\bar{l}_n)$
such that
\begin{itemize}
\item $\mathcal{L}$ is a principal $L$-bundle on $\pone$ such that $\mathcal{L}\times^L \Bbb{C}_{-\omega_i}$ has degree $a_{i}$ for each $\alpha_i\in S_P$.
\item For $k=1,\dots,n$, $\bar{l}_k\in \mathcal{L}_{b_k}/B_L$.
\end{itemize}

There is a canonical morphism of stacks $\phi:\Parbun_L\to \Parbun_P$. Similar to the definition of the theta bundle $\theta$ on
$\Parbun_P$, we can define the theta bundle $\theta'$ on
$\Parbun_L$. From the functoriality of the theta bundles, it is easy to see that
$$\phi^*(\theta)=\theta'.$$
We have the following crucial definition.
\begin{defi}
We call $(u_1,\dots,u_n;d)$  {\it quantum Levi-movable} if $\theta'$ does not  vanish identically on $\Parbun_L$.
\end{defi}
We have the following key proposition.
\begin{proposition}\label{clef}
Consider a quasi-parabolic $L$-bundle  $\tilde{\mathcal{L}}= (\mathcal{L};\bar{l}_1,\dots,\bar{l}_n)$
 of degree $d$. Let $\phi(\tilde{\mathcal{L}})=\tilde{\mathcal{P}}=(\mathcal{P},\bar{l}_1,\dots,\bar{l}_n)$ be the corresponding point of $\Parbun_P$. Then, the central one parameter subgroup $t^{\bar{x}_i}$ of $L$  corresponding to $\bar{x}_i,\  \alpha_i \in S_P$, acts on
$D(\mk(\tilde{\mathcal{P}}))$ by multiplication by $t^{\mu_i}$, where
$$\mu_i= (\chi_{e}-\sum_{k=1}^n\chi_{u_k})(\bar{x}_i) + \sum _{\alpha\in R^+\setminus R^+_{\fl}} \alpha(\bar{x}_i)\alpha(\tilde{d}),$$
$\bar{x}_i$ is defined by the equation \eqref{eqn0'}, and,  as in Theorem \ref{one},  $\tilde{d}=\sum_{\alpha_j\in S_P} a_{j}\alpha_j^\vee.$
\end{proposition}
\begin{proof}
Note that $$D(\mk)=D(\mathcal{P}\times^P T_{\dot{e}})\tensor \prod_{k=1}^n D(i_{b_k*}\frac{\mathcal{P}_{b_k}\times^P T_{\dot{e}}}{T(\bar{l}_k, u_k, b_k)})^*.$$

It is easy to see that $t^{\bar{x}_i}$ acts on $D(i_{b_k*}\frac{\mathcal{P}_{b_k}\times^P T_{\dot{e}}}{T(\bar{l}_k, u_k, b_k)})^*$ by $t^{-\chi_{u_k}(\bar{x}_i)}.$
We next calculate the action of  $t^{\bar{x}_i}$ on $D(\mathcal{P}\times^P T_{\dot{e}})$:

Let $\mathcal{V}$ be a vector bundle on  $\pone$, and let $T:\mv\to \mv$  be the  multiplication
by the scalar $c^{-1}$ on fibers. Then, clearly,  $T$ acts on $D(\mathcal{V})$ by the scale $c$ raised to the exponent $$\chi(\pone,\mv)=\rk \mv +\deg \mv.$$

Suppose the vector bundle  $\mathcal{P}\times^P T_{\dot{e}}$ is filtered with the associated graded pieces being the  vector bundles $\mv_1,\dots,\mv_s$, such that
$t^{\bar{x}_i}$  acts on $\mv_r$ by the scale $t^{-\gamma_r}$. Then, $t^{\bar{x}_i}$ acts on $D(\mathcal{P}\times^P T_{\dot{e}})$ via
$$t^{\sum_r(\rk\mv_r +\deg\mv_r)\gamma_r}.$$
This allows us to reduce the calculation of the action of $t^{\bar{x}_i}$ on $D(\mathcal{P}\times^P T_{\dot{e}})$  using the filtration of $T_{\dot{e}}$ given in Subsection \ref{formules}.
The desired $\mv_r$ are the quotients $T_{i,r}/T_{i,r-1}$.  Now, use the formula \eqref{eqn2'}.

Finally, it is  easy to see that $\sum_r\rk\mv_r=\chi_e(\bar{x}_i)$. Combining these, we get the proposition.
\end{proof}

\begin{theorem}\label{two}
The following are equivalent:

(a)  $(u_1,\dots,u_n;d)$ is quantum Levi-movable.

(b)  $\langle \sigma^P_{u_1},\sigma^P_{u_2},\dots,\sigma^P_{u_n}\rangle_d\neq 0$ and
\begin{equation}\label{vivid}
\forall \alpha_i\in S_P,\ (\chi_{e}-\sum_{k=1}^n\chi_{u_k})(x_i) + \sum _{\alpha\in R^+\setminus R^+_{\fl}}\alpha(x_i)\alpha(\tilde{d})=0.
\end{equation}

(c) $\langle \sigma^P_{u_1},\sigma^P_{u_2},\dots,\sigma^P_{u_n}\rangle_d^{\qdot_0}\neq 0$, where $\langle \sigma^P_{u_1},\sigma^P_{u_2},\dots,\sigma^P_{u_n}\rangle_d^{\qdot_0}$ is, by definition, equal to the coefficient of $q^d\sigma^P_{w_ou_nw_o^P}$ in $\sigma^P_{u_1}\qdot_0 \dots \qdot_0\sigma^P_{u_{n-1}}.$
\end{theorem}

\subsection{Proofs of Theorems \ref{one} and \ref{two}}\label{section3.6}
 Fix an element $x=\sum_{\alpha_i\in S_P}\,d_i\bar{x}_i\in \fh$ such that each $d_i$ is a strictly positive integer. Then, $t^{x}$ is a central one parameter subgroup of $L$; in particular, it is contained in $B_L$. For any $t\in \Bbb{G}_m$,  define the conjugation $\phi_t:P\to P, p\mapsto
t^xpt^{-x}$. This extends to a group homomorphism $\phi_0:P \to L\subset P$, giving rise to a regular map $\phi:P\times\Bbb{A}^1\to P$,
extending the map $(p,t)\mapsto \phi_t(p)$, for  $p\in P$ and $t\in \Bbb{G}_m$. Clearly, ${\phi_t}_{|L}=I_L$, for all $t\in \Bbb{A}^1$, where
$I_L$ is the identity map of $L$.

 Let $\mathcal{P}$ be a principal $P$-bundle.  Define a family of principal $P$-bundles $\mathcal{P}_t$ parameterized by $t\in \Bbb{A}^1$, where
 $\mathcal{P}_t$ is the principal $P$-bundle induced by $\mathcal{P}$ via $\phi_t$, i.e., $\mathcal{P}_t=\mathcal{P}\times^{P,\phi_t} P$.
Since, the image of $\phi_0$ is contained in $L$, we get a principal $L$-bundle $\mathcal{L}$ from $\mathcal{P}$ via the homomorphism
$\phi_0$. Clearly, $\mathcal{P}_0 = \mathcal{L}\times^L  P$.
 We write $\Gr(\mathcal{P})=\mathcal{L}\times^L  P$. We will refer to this as the {\it Levification process}. So, we have found a degeneration of $\mathcal{P}$ to $\Gr(\mathcal{P})$ parameterized by $t\in \Bbb{A}^1$.

If $\bar{p}_k\in\mathcal{P}_{b_k}/B_L$, we canonically have $\bar{p}_k(t)\in (\mathcal{P}_t)_{b_k}/B_L$, for any $t\in \Bbb{A}^1$, defined as
$\bar{p}_k(t) = {\phi_t}_* (\bar{p}_k)$.  At $t=0$, the image is in $\mathcal{L}_{b_k}/B_L$.

We therefore have a  $\gmm$-equivariant line bundle $D(\mathcal{K}_t)$ on $\Bbb{A}^1$. Furthermore, we have a $\gmm$-equivariant section $\theta(\mk_t)$ of $D(\mathcal{K}_t)$. The following statement is immediate
(cf. Proposition 10 in \cite{Belkale-Kumar}).

\begin{lemma}\label{oldnew}
Let $\gmm$ act  on $\Bbb{A}^1$ by multiplication. Let $\mar$ be a $\gmm$-equivariant line bundle on
$\Bbb{A}^1$ with a $\gmm$-invariant section $s$. Suppose $\gmm$ acts on the fiber $\mar_0$ over $0$ by multiplication by $t^{\mu},\mu\in\Bbb{Z}$, i.e., the Mumford index $\mu^\mar(t,\lambda_o)=\mu$, for any $t\in \Bbb{A}^1$, where $\lambda_o$ is the one parameter subgroup $z\mapsto z$.
Then,
\begin{enumerate}
\item[(a)] If $s\neq 0$, then $\mu\geq 0$.
\item[(b)] If $\mu=0$ and $s\neq 0$, then $s(0)\neq 0\in\mar_0$.
\end{enumerate}
\end{lemma}
\subsubsection{Proof of Theorem \ref{one}}
Assume that $\langle \sigma^P_{u_1},\sigma^P_{u_2},\dots,\sigma^P_{u_n}\rangle_d\neq 0$. Then, by Lemma \ref{wash}, the section
$\theta\in H^0(\Parbun_P, \mar)$ is nonzero. Let $\tilde{\mathcal{P}}=(\mathcal{P};\bar{p}_1,\dots,\bar{p}_n)$ be a quasi-parabolic $P$-bundle in $\Parbun_P$ such that $\theta(\tilde{\mathcal{P}}) \neq 0$. Considering the one parameter degeneration ${\tilde{\mathcal{P}}}_t$ as above and using Lemma \ref{oldnew} and Proposition \ref{clef}, we get that $\sum_{\alpha_i\in S_P}\,d_i\mu_i\geq 0$, for any strictly positive integers $d_i$, where
$$\mu_i= (\chi_{e}-\sum_{k=1}^n\chi_{u_k})(\bar{x}_i) + \sum _{\alpha\in R^+\setminus R^+_{\fl}} \alpha(\bar{x}_i)\alpha(\tilde{d}).$$
From this we conclude that each $\mu_i\geq 0$.  This proves Theorem \ref{one}.

\subsubsection{Proof of Theorem \ref{two}} We first prove $(a) \Longrightarrow  (b)$:

Take a quasi-parabolic $L$-bundle $\tilde{\mathcal{L}}=(\mathcal{L};\bar{l}_1,\dots,\bar{l}_n) \in \Parbun_L (d)$ such that
$\theta'(\tilde{\mathcal{L}})\neq 0$. Let $\tilde{\mathcal{P}}=\phi(\tilde{\mathcal{L}})= (\mathcal{P};\bar{l}_1,\dots,\bar{l}_n)$
 be the corresponding point of $\Parbun_P$. Hence,  $\theta$ does not vanish at the quasi-parabolic $\tilde{\mathcal{P}}$.
The right multiplication by the $L$-central one parameter subgroups $t^{\bar{x}_i}$, for $\alpha_i\in S_P$, induces an automorphism of the
quasi-parabolic $L$-bundle $\tilde{\mathcal{L}}$  and hence that of $\tilde{\mathcal{P}}$. These should act trivially on
$\theta(\tilde{\mathcal{P}})$, and hence  if $\theta(\tilde{\mathcal{P}})\neq 0$, we get that  $t^{\bar{x}_i}$ acts trivially on
$D(\mk(\tilde{\mathcal{P}}))$ (cf. $\S$ \ref{detnew}). Hence, we obtain that (a) implies \eqref{vivid} by using Proposition \ref{clef}. Further, by Lemma \ref{wash}, we get that $\langle \sigma^P_{u_1},\sigma^P_{u_2},\dots,\sigma^P_{u_n}\rangle_d\neq 0$. This proves (b).

For the reverse direction,  by Lemma \ref{wash}, assume that $\theta$ is non-vanishing on a quasi-parabolic $\tilde{\mathcal{P}}$. Performing the above degeneration,
we find the desired conclusion using Lemma \ref{oldnew} (b) and Proposition \ref{clef}. The equivalence of (b) and (c) follows from the definition of
$\qdot_0$. This proves Theorem \ref{two}.

\section{Determination of the multiplicative eigen polytope in terms of the deformed quantum cohomology}\label{sufficiency}

Let $G$ be a simple, connected, simply-connected complex algebraic group.

Consider the {\em fundamental alcove} $\mathscr{A}\subset
\mathfrak{h}$ defined by
$$
\mathscr{A}=\left\{\mu \in \mathfrak{h}:\alpha_{i}(\mu)\geq 0\text{~
  and~ } \theta(\mu)\leq 1\right\},
$$
where $\theta$ is the highest root of $\mathfrak{g}$. Then,
$\mathscr{A}$ parameterizes the $K$-conjugacy classes of $K$ under the
map $C:\mathscr{A}\to K/\Ad K$,
$$
\mu \mapsto c(\Exp (2\pi i\mu)),
$$
where $K$ is a maximal compact subgroup of $G$ and $c(\Exp (2\pi
i\mu))$ denotes the $K$-conjugacy class of $\Exp(2\pi i\mu)$. Fix a
positive integer $n\geq 2$ and define the {\em multiplicative eigen polytope}
$$
\mathscr{C}_{n} := \left\{(\mu_{1},\ldots,\mu_{n})\in \mathscr{A}^{n}:
1\in C(\mu_{1})\dots C(\mu_{n})\right\}.
$$

Then, it is known that $\mathscr{C}_{n}$ is a rational convex
polytope with nonempty interior in $\mathfrak{h}^n$ (cf. \cite[Corollary 4.13]{MW}). Our aim is to describe the facets (i.e., the codimension one
faces) of $\mathscr{C}_{n}$.

The following theorem is one of our main results. In the case $G=\SL_2$, it was proved by Biswas \cite{Bi}. For $G=\SL_m$, it was proved by Belkale
\cite{B1} (and a slightly weaker result by Agnihotri-Woodward \cite{AW}).  (Observe that for $G=\SL_m$, by Lemma \ref{minuscule}, the deformed quantum cohomology coincides with the quantum cohomology of $G/P$ for maximal $P$.)
\begin{theorem}\label{main}
Let $(\mu_{1},\ldots,\mu_{n})\in \mathscr{A}^{n}$. Then, the following
are equivalent:

(a) $(\mu_{1},\ldots,\mu_{n})\in \mathscr{C}_{n}$,

(b) For any standard maximal parabolic
subgroup $P$ of $G$, any $u_{1},\ldots,u_{n}\in W^{P}$, and any $d\geq
0$ such that the deformed small quantum cohomology (Gromov-Witten) invariant
(cf. Definition \ref{qdeformquantum} and Theorem \ref{two} (c))
$$
\langle \sigma^{P}_{u_{1}},\ldots, \sigma^{P}_{u_{n}} \rangle_{d}^{\qdot_{0}}=1,
$$
the following inequality is satisfied:
$$
\mathscr{I}^{P}_{(u_{1},\ldots, u_{n};d)}:\medskip\medskip\medskip\,\,\,\,\,\,\,\,\,\,\sum^{n}_{k=1}
\omega_P(u_k^{-1}\mu_k) \leq d,
$$
where $\omega_{P}$ is the fundamental weight $\omega_{i_{P}}$ such
that $\alpha_{i_P}$ is the unique simple root in $S_{P}$.
\end{theorem}

\begin{proof}
{ (a) $\Rightarrow$ (b):} In fact, as proved in \cite{Teleman-W}, for any
$u_{1},\ldots,u_{n}\in W^{P}$ and any $d\geq 0$ such that the
tuple $(P;u_1,\dots,u_n; d)$ is quantum non-null (see Definition \ref{notnull}), the
inequality $\mathscr{I}^{P}_{(u_{1},\ldots,u_{n};d)}$ is
satisfied. We include a proof for completeness. Since $\mathscr{C}_{n}$ is a rational polytope with
nonempty interior in $\mathfrak{h}^{n}$, we can assume that each
$\mu_k$ is regular (i.e., each $\alpha_i(\mu_k)>0$), rational (i.e., each $\alpha_i(\mu_k)\in \mathbb{Q}$) and $\theta(\mu_{k})<1$.
As in earlier sections, fix
distinct points $b_{1},\ldots,b_{n}\in \mathbb{P}^{1}$ and
let $\mathbb{M}_{G}(\vec{\mu})$ be the parabolic moduli
space of parabolic semistable principal $G$-bundles over
$\mathbb{P}^{1}$ with parabolic weights $\vec{\mu}=(\mu_{1},\ldots,\mu_{n})$
associated to the points $(b_{1},\ldots,b_{n})$ respectively. We follow the version of
$\mathbb{M}_{G}(\vec{\mu})$ from \cite[$\S$2.6]{Teleman-W}.

By the generalization of the Mehta-Seshadri theorem \cite{MS} to arbitrary groups
 (cf. \cite[Theorem 3.3]{Teleman-W}; \cite{BR}), the
assumption (a) is equivalent to the assumption that the moduli space
$\mathbb{M}_{G}(\vec{\mu})$ is nonempty. By
 \cite[Proposition 4.2]{Teleman-W},
  $\mathbb{M}_{G}(\vec{\mu})$ is nonempty if and only if
  the trivial bundle $\epsilon_{G}:\mathbb{P}^{1}\times G\to
  \mathbb{P}^{1}$ with general parabolic structures at the marked
  points $b_{1},\ldots,b_{n}$ is parabolic semistable.

Let $u_{1},\ldots, u_{n}\in W^{P}$ and $d\geq 0$ be such that
$(P;u_1,\dots,u_n; d)$ is quantum non-null.
Hence, there exists a morphism $f:\mathbb{P}^{1}\to X^P$ of degree $d$ such
that $f(b_k)\in g_{k}C^{P}_{u_k}$, for general
$g_{1},\ldots,g_{n}\in G$, where $C^{P}_{u_k}:=Bu_kP/P$ is the
Schubert cell. Since $(g_{1},\ldots,g_{n})\in G^{n}$ are general, we
can assume that the trivial bundle $\epsilon_{G}$ with parabolic
structure $g_{k}B$ at $b_k$ and weights $\mu_k$ is parabolic semistable (we
have used the assumption that each $\mu_k$ is regular). In
particular, for the parabolic reduction $\sigma : \mathbb{P}^{1}\to
\epsilon_{G}/P$ induced from the morphism $f:\mathbb{P}^{1}\to X^P$, we get
(from the definition of the parabolic semistability):
$$
\deg (\sigma^{*}(\epsilon_{G}(-\omega_{P})))+\sum^{n}_{k=1}
\omega_{P}(u_k^{-1}\mu_k) \leq 0,
$$
since $u_k\in W^{P}\simeq W/W_{P}$ is the relative position of
$\sigma(b_k)=f(b_k)$ and $g_{k}B$, where the line bundle $\epsilon_{G}(-\omega_{P})$ over $\epsilon_G/P$ is defined in 
Subsection \ref{subsection3.1}. But, by the definition of degree,
$$
d:= \deg(\sigma^{*}(\epsilon_{G}(\omega_{P}))).
$$

Hence, we get
$$
\sum^{n}_{k=1}\omega_P(u_k^{-1}\mu_k) \leq d.
$$

This proves the (b)-part. We prove the implication `(b) $\Rightarrow$ (a)' below.
\end{proof}
\begin{remark} \label{nonnull}
As shown by \cite{Teleman-W}, the above argument shows that the inequalities
$\mathscr{I}^{P}_{(u_{1},\ldots,u_{n}; d)}$, in fact, determine the
polytope $\mathscr{C}_{n}\subset \mathscr{A}^{n}$ provided we run $P$ through standard
maximal parabolic subgroups and $(u_{1},\ldots,u_{n};d)\in
(W^{P})^{n}\times \mathbb{Z}_{+}$ such that
$(P;u_1,\dots,u_n; d)$ is  quantum non-null.
\end{remark}

Before we come to the proof of the implication `(b) $\Rightarrow$ (a)' in Theorem \ref{main}, we need to review  the canonical reduction of
parabolic $G$-bundles.

\subsection{Canonical reduction of parabolic $G$-bundles.} \label{equivariant}
Let us fix parabolic weights $\vec{\mu}=(\mu_1, \dots, \mu_n) \in \mathscr{A}^n$ associated to the points $b_1, \dots, b_n\in \pone$ respectively.
We further assume that each $\mu_k$ is rational, regular and $\theta (\mu_k)<1$. Fix a positive integer $N$ such that  
$N\mu_k$ belongs to the coroot lattice of $G$ for all $k$.  Let $\Gamma:=\mathbb{Z}/(N)$ be the cyclic group of order $N$. We fix a generator
$\gamma_o\in \Gamma$. Now, take an irreducible  smooth curve $C$ with an action of $\Gamma$ on $C$ and a $\Gamma$-equivariant morphism $\pi:C\to \pone$ (with the trivial action of $\Gamma$ on $\pone$) satisfying the following:

(a) $\pi^{-1}(b_k)$ is a single point $\tilde{b}_k$, for all $k$,

(b) $\Gamma$ acts freely on a nonempty  open subset of $C$, and

(c) the map $\pi$ induces an isomorphism $ C/\Gamma \simeq \pone.$

Following Teleman-Woodward \cite[Section 2.2]{Teleman-W}, for any $\Gamma$-equivariant $G$-bundle $E$ on $C$, such that at the points $\tilde{x}_k$, the generator of $\Gamma$ acts via the conjugacy class of $\Exp (2\pi i \mu_k)$ , we construct a
quasi-parabolic $G$-bundle $\tilde{\mathcal{E}}=(\E; \bar{g}_1, \dots, \bar{g}_n)$ on $\pone$ as follows.

For simplicity, we give the construction in the analytic category; the construction in the algebraic category is similar.

Choose a small enough analytic open neighborhood $U_k$ of $b_k$ in $\pone$ and a coordinate $z$ in $\tilde{U}_k:=\pi^{-1}(U_k)$ such that
the map $\pi: \tilde{U}_k\to U_k$ is given by $z\mapsto z^N$ and, moreover, the action of the generator $\gamma_o\in \Gamma$ on  $\tilde{U}_k$ is given by $z\mapsto e^{2\pi i/N}z$. Moreover, by \cite[Section 11]{HK}, we can choose  $U_k$ small enough so that there is a
$\Gamma$-equivariant analytic isomorphism $\theta_k:  E_{|\tilde{U}_k} \to \tilde{U}_k\times G$ such that $\gamma_o$ acts on $\tilde{U}_k\times G$ via
$$\gamma_o(z, g)= (e^{2\pi i/N}z, \Exp(2\pi i\mu_k)g).$$
Let $E^{-N\mu_k}$ denote the set of $\Gamma$-invariant meromorphic sections $\sigma:\tilde{U}_k\to E_{|\tilde{U}_k}$ such that
$ z^{-N\mu_k}\cdot \sigma$ is regular on $\tilde{U}_k$.  Then, $\sigma_k:\tilde{U}_k\to E_{|\tilde{U}_k}$, given by
$\sigma_k(z)=(z, z^{N\mu_k})$ (under the above isomorphism $\theta_k$) is a section contained in $E^{-N\mu_k}$.

As in  \cite[Section 2.2]{Teleman-W}, there is a principal $G$-bundle $\E$ over $\pone$ isomorphic to $\Gamma\backslash E$ over
$C\setminus \{\tilde{b}_, \dots, \tilde{b}_n\}$ and such that $E^{-N\mu_k}$ is the set of sections of $\E$ over $U_k$. Moreover, the section
$\sigma_k$ evaluated at $b_k$ provides a parabolic reduction $\bar{g}_k$ of the fiber $\E_{b_k}/B$. (We have used here the assumption that $\mu_k$'s are
 regular.) Thus, for any
$\Gamma$-equivariant $G$-bundle $E$ on $C$, such that at the points $\tilde{x}_k$, the generator of $\Gamma$ acts via the conjugacy class of $\Exp (2\pi i \mu_k)$, we have constructed a
quasi-parabolic $G$-bundle $\tilde{\mathcal{E}}=(\E; \bar{g}_1, \dots, \bar{g}_n)$ on $\pone$.

Let $P$ be a standard parabolic subgroup of $G$. From the above construction, it is clear that any $\Gamma$-equivariant principal $P$-subbundle of $E$ canonically gives rise to a $P$-subbundle of  $\E$.

Let  $\Bun_G^\Gamma(C)=\Bun_G^{\Gamma,\vec{\mu}}(C)$ be the moduli stack of $\Gamma$-equivariant principal $G$-bundles on $C$, such that at the points $\tilde{x}_k$, the generator of $\Gamma$ acts via the conjugacy class of $\Exp (2\pi i \mu_k)$.

For a parabolic reduction $E_P$ of $E\in \Bun_G^\Gamma(C)$ to  $P$, we have the notion of degree $\deg (E_P):=(a_i)_{\alpha_i\in S_P}\in H_2(X^P, \mathbb{Z})$, where $a_i$ is the degree of $E_P\times^P\mathbb{C}_{-\omega_i}$. Similarly, for a parabolic reduction $\tilde{\E}_P$
of a quasi-parabolic $G$-bundle $\tilde{\mathcal{E}}=(\E; \bar{g}_1, \dots, \bar{g}_n)$ on $\pone$, one defines the parabolic degree
$$\pardeg (\tilde{\E}_P):= (b_i)_{\alpha_i\in S_P}\in H_2(X^P, \mathbb{Z}),$$
 where $b_i:= \deg (E_P\times^P\mathbb{C}_{-\omega_i})+\sum_{k=1}^n\omega_i(u_k^{-1}\mu_k)$ and $u_k\in W^P$ is the relative position of
$\E_P(b_k)$ and $\bar{g}_k$.

We summarize this correspondence in the following result due to Teleman-Woodward \cite[Theorem 2.3]{Teleman-W}.
\begin{theorem}  \label{equivalence} There is an isomorphism of stacks:
$$\Bun_G^\Gamma(C)\to \Parbun_G$$
taking $E\mapsto \tilde{\E}$, where
$\Parbun_G$ is as defined in Subsection \ref{space}.

Moreover, the $\Gamma$-equivariant reductions of any $E\in \Bun_G^\Gamma(C)$ to a parabolic subgroup $P$ of $G$  correspond bijectively to the reductions of $\tilde{\E}$ to $P$.

Further, for any $\Gamma$-equivariant reduction $E_P$ of $E$,
\begin{equation}\label{degeqn} \deg (E_P)= N\pardeg (\tilde{\E}_P),
\end{equation}
where $\tilde{\E}_P$ is the corresponding reduction of $\tilde{\E}$.
\end{theorem}

\begin{defi} \label{canonical} Let $\tilde{\E}$ be a quasi-parabolic $G$-bundle on $\pone$ with parabolic weights $\vec{\mu}$ assigned to the marked points $b_1, \dots, b_n$ as in Subsection \ref{equivariant}. Then, a reduction $\sigma_P(\tilde{\E})$ of $\tilde{\E}$ to a standard parabolic subgroup $P$ is called
{\it parabolic canonical} if the corresponding $\Gamma$-equivariant parabolic reduction $\sigma_E$ of the corresponding ($\Gamma$-equivariant) $G$-bundle $E$ over $C$ is canonical.
(Observe that, by the uniqueness of the canonical reduction of any $G$-bundle over $C$ as in \cite[Theorem 4.1]{BH},  $\sigma_E$ is unique; in particular, from the $\Gamma$-equivariance of $E$, we get that $\sigma_E$ is $\Gamma$-equivariant.)

 By the uniqueness of $\sigma_E$, we get that the parabolic canonical reduction  $\sigma_P(\tilde{\E})$ of $\tilde{\E}$ is unique. In fact,  $\sigma_P(\tilde{\E})$ does not admit any infinitesimal deformations either. More precisely, we have the following result:
\end{defi}

\begin{lemma} \label{reduced} Let $\tilde{\E}=(\E; \bar{g}_1, \dots, \bar{g}_n)$ be a quasi-parabolic $G$-bundle on $\pone$ with parabolic weights $\vec{\mu}$ assigned to the marked points $b_1, \dots, b_n$ and let  $\sigma_P(\tilde{\E})$ be its canonical reduction (to the parabolic subgroup $P$) of degree
$d$. Let $u_k \in W^P$ be the relative position of $\sigma_P(\tilde{\E})(b_k)$ and $\bar{g}_k$, for any $1\leq k \leq n$. Recall that, by the definition of $Z_d'(\tilde{\mathcal{E}})=Z_d'(\tilde{\mathcal{E}}; u_1, \dots, u_n)$ as in Subsection \ref{space}, $\sigma_P(\tilde{\E})\in Z_d'(\tilde{\mathcal{E}})$.
Then, $\sigma_P(\tilde{\E})$ is the unique point of  $Z_d'(\tilde{\mathcal{E}})$ and, moreover, it is a reduced point, i.e., the Zariski tangent space
$$T(Z_d'(\tilde{\mathcal{E}}))_{\sigma_P(\tilde{\E})}=0.$$
\end{lemma}
\begin{proof} To prove that the scheme  $Z_d'(\tilde{\mathcal{E}})$ contains the unique closed point $\sigma_P(\tilde{\E})$, use the fact that the canonical reduction $\sigma_E$ of a $G$-bundle on $C$ is completely characterized by the degree (cf. \cite[Proposition 3.1]{BH}).  Now, use the correspondence
as in Theorem \ref{equivalence} and the degree comparisons as in  the equation \eqref{degeqn}.

We now prove that  $\sigma_P(\tilde{\E})$ is a reduced point of $Z_d'(\tilde{\mathcal{E}})$. If  $Z_d'(\tilde{\mathcal{E}})$ were to possess any infinitesimal deformations at $\sigma_P(\tilde{\E})$, then so would the canonical reduction   $\sigma_E$ on the curve $C$.
 But, according to a result of Heinloth
\cite[Theorem 1]{Hei}, the canonical reduction of a principal $G$-bundle on a smooth curve does not have any infinitesimal deformations. We are therefore done by the equivalence of stacks as in Theorem \ref{equivalence}.
\end{proof}
Now, we are ready to prove the other direction of Theorem \ref{main}.
\begin{proof}
{(b) $\Rightarrow$ (a) in Theorem \ref{main}:}
Let $\mathscr{I}^{P}_{(u_{1},\ldots,u_{n};d)}$ be a facet of
$\mathscr{C}_{n}$ for some standard maximal parabolic subgroup $P$ and
$(u_{1},\ldots, u_{n};d)\in (W^{P})^{n}\times \mathbb{Z}_{+}$
such that
$ (P; u_1, \dots, u_n;d)$ is quantum non-null.
Any facet of $\mathscr{C}_{n}$ is of this form by Remark \ref{nonnull}.

Take a rational regular element $\vec{\mu}=(\mu_{1},\ldots,\mu_{n})\in
\mathscr{A}^{n}$ (i.e.,  each $\mu_k$ is rational, regular and $\theta(\mu_k)<1$) such
that $\vec{\mu}$ does {\em not} satisfy the inequality
$\mathscr{I}^{P}_{(u_{1},\ldots,u_{n};d)}$, but satisfies
every other inequality defining the other facets of
$\mathscr{C}_{n}$. Since $\vec{\mu}\not\in \mathscr{C}_{n}$, the
corresponding parabolic moduli space $\mathbb{M}_{G}(\vec{\mu})$ is empty. Hence,
the trivial bundle $\epsilon_{G}$ with any parabolic structure
$\overline{g}_{k}=g_{k}B$ at $b_k$ (and weights $\mu_k$) is not
semistable. We fix some general elements $\overline{g}_{k}\in G/B$ and
consider the parabolic bundle (with weights $\vec{\mu}$)
$
\tilde{\E}=(\E=\epsilon_{G},\overline{g}), \text{~ where~ } \overline{g}:=(\overline{g}_{1},\ldots,\overline{g}_{n}).
$
Since $\tilde{\E}$ is unstable (i.e., non-semistable), the
canonical parabolic reduction $\sigma_Q (\tilde{\E})$ (to a standard parabolic
subgroup $Q$) does not satisfy the semistability inequality for $\vec{\mu}$,
i.e., for some maximal parabolic subgroup $Q'\supset Q$,
$$
-\deg
\left(\sigma_Q ({\tilde{\E}})^*({\E}(\omega_{Q'}))\right)+\sum^{n}_{k=1}
\omega_{Q'} (v_k^{-1}\mu_k) >0,
$$
where $v_{k}\in W/W_{Q'}$ is the relative position of
$\sigma_Q (\tilde{\E})(b_k)$ and $\bar{g}_k$. Moreover, we can assume
that $\mathscr{I}^{Q'}_{(v_{1},\ldots,v_{n};d')}$ defines a facet of
$\mathscr{C}_{n}$, where $d':=\deg
(\sigma_Q ({\tilde{\E}})^*({\E}(\omega_{Q'})))$. By our choice
of $\vec{\mu}$, among the inequalities coming from the facets, since $\vec{\mu}$
only violates the inequality
$\mathscr{I}^{P}_{(u_{1},\ldots,u_{n};d)}$, we get the following:
\begin{itemize}
\item[(A$_{1}$)] $Q'=P$ and hence $Q=P$,

\item[(A$_{2}$)] each $v_{k}=u_k$, and

\item[(A$_{3}$)] $d'=d$.
\end{itemize}

We next claim that
\begin{equation}\label{eq4.1}
\langle \sigma^{P}_{u_{1}},\ldots,\sigma^{P}_{u_{n}}\rangle_{d}=1:
\end{equation}

By (A$_{1}$) - (A$_{3}$), the parabolic degree  of
$\sigma_P (\tilde{\E})$ is given by
$$
\pardeg (\sigma_P (\tilde{\E}))= -d+\sum^{n}_{k=1} \omega_{P} (u_k^{-1}\mu_k).
$$

 Since the canonical reduction
$\sigma_P (\tilde{\E})$ is the only reduction of $\tilde{\E}$ to $P$
with the parabolic degree that of  $\sigma_P (\tilde{\E})$ (cf. Lemma \ref{reduced}), we get that $\langle
\sigma^{P}_{u_{1}},\ldots,\sigma^{P}_{u_{n}}\rangle_{d}=1$.
This proves \eqref{eq4.1} (which was already proved in \cite{Teleman-W}). We now come to the
new part which is that
\begin{equation}\label{desired}
\langle \sigma^{P}_{u_{1}},\ldots,\sigma^{P}_{u_{n}}\rangle^{\qdot_{0}}_{d}=1,
\end{equation}
i.e., in view of \eqref{eq4.1} and Theorem \ref{two}, $(u_{1},\ldots, u_{n};d)$ is
quantum Levi-movable:

Choose  a trivialization $\bar{p}_k$ of $\sigma_P (\tilde{\E})_{b_k}/B_L$ so that  $\bar{g}_kBu_kP= \bar{p}_k\Lambda_{u_k}$.
(Observe that $\bar{p}_k$ is unique modulo the stabilizer of $\Lambda_{u_k}$.)
Consider the family $\tilde{\sigma}_P (\tilde{\E})_t = (\sigma_P (\tilde{\E})_t; \bar{p}_1(t), \dots, \bar{p}_n(t))$ of quasi-parabolic $P$-bundles
on $\mathbb{P}^1$ parameterized by $t\in \mathbb{A}^1$, defined in $\S$\ref{section3.6},
where we take $x=\bar{x}_P$.

It is easy to see that, as  quasi-parabolic  $P$-bundles,
\begin{equation}
 \tilde{\sigma}_P (\tilde{\E})_t\simeq \tilde{\sigma}_P (\tilde{\E}), \text{~ for~ }
  t\neq 0.
\end{equation}

Let $\tilde{\E}_{0}$ denote the parabolic $G$-bundle obtained from
$\sigma_P (\tilde{\E})_0$ via the extension of the structure group
$P\hookrightarrow G$ and the parabolic structures $p_1(0)u_1^{-1}B, \dots , p_n(0)u_n^{-1}B$ at the points $b_1, \dots, b_n$ on $\pone$
 (and with the same parabolic weights $\mu_k$ at $b_k$). Then, we assert that the canonical reduction of $\tilde{\E}_{0}$ coincides with
$\tilde{\sigma}_P (\tilde{\E})_0.$

To prove this, observe first that for any parabolic $G$-bundle $\tilde{\E}$ on $\pone$ that corresponds to a $\Gamma$-equivariant principal $G$-bundle
$E$ on $C$ (via  the correspondence of Theorem \ref{equivalence}) and any reduction $E_P$ of  $E$ to a parabolic subgroup $P$ and the corresponding reduction $\tilde{\E}_P$, the deformations $(E_P)_t$ and $(\tilde{\E}_P)_t$ over  $t\in \mathbb{A}^1$ (defined in Subsection
\ref{section3.6} for $x=\bar{x}_P$) correspond.
Further, from the characterization of the canonical reduction of (non-parabolic) $G$-bundles on $C$ as in \cite[Proposition 3.1]{BH},
we see that $\sigma_P(E)_0$ is the canonical reduction of $E_0$,  where $E_0$ is the $G$-bundle on $C$ obtained from
 $\sigma_P(E)_0$ via the extension of the structure group $P \hookrightarrow G$. Combining these, we get that the canonical reduction of $\tilde{\E}_{0}$ coincides with
$\tilde{\sigma}_P (\tilde{\E})_0.$

Finally, we come to the proof of the equation \eqref{desired}:

By the last assertion, the quasi-parabolic $G$-bundle $\tilde{\E}_{0}$ comes from a quasi-parabolic $L$-bundle $\tilde{\ml}$ of degree $d$ by extension
of the structure group $L\hookrightarrow G$. Further, by Lemma \ref{reduced}, $Z_d'(\tilde{\E}_{0})$ is a reduced scheme of dimension $0$ at
$\tilde{\ml}\times^L G$. Thus, $\theta'(\tilde{\ml})\neq 0.$ This proves  the equation \eqref{desired} and hence Theorem \ref{main} is proved.
\end{proof}

\section{Eigenvalue problem and global sections of line bundles on moduli spaces}\label{ev}

Recall the definition of the quasi-parabolic moduli stack  $\Parbun_G$ from Subsection \ref{space}.
\begin{defi}
For any weight $\lambda\in X(H)$ and  $1\leq k\leq n$, define the line bundle $\mathcal{L}_k(\lambda)$ over $\Parbun_G$, that assigns
to $\tilde{\me}=(\mathcal{E};\bar{g}_1,\dots,\bar{g}_n)\in \Parbun_G$ the line which is the fiber of the line bundle $\mathcal{E}_{b_k}\times^B \Bbb{C}_{-\lambda} \to
\mathcal{E}_{b_k}/B$ over the point $\bar{g}_k$.

 Therefore,  given weights $\vec{\lambda}=(\lambda_1,\dots,\lambda_n)$ and an integer $m$, we can form the line bundle 
$\mathcal{M}(\vec{\lambda},m)$ over $\Parbun_G$ defined by
$$\mathcal{M}(\vec{\lambda},m):=\bD^m\otimes (\tensor_{k=1}^n \mathcal{L}_k(\lambda_k)),$$
where the determinant line bundle $\bD$ over $\Parbun_G$ assigns to $\tilde{\me}$ the determinant of cohomology line
$D(\mathcal{E}\times^G \frg)$.

This line bundle should be considered to be at level $2g^*m$, where $g^*$ is the dual Coxeter number (defined in Lemma \ref{coxeter}) of $\mathfrak{g}$
(cf. \cite[Theorem 5.4 and Lemma 5.2]{KNR}).
\end{defi}

A dominant integral weight $\lambda$ is said to be of {\it  level} a nonnegative integer $d$ if $\lambda(\theta^\vee)\leq d$.

\begin{theorem}\label{evproblem} Let $m>0$ and let $\vec{\lambda}=(\lambda_1,\dots,\lambda_n)$ be dominant regular integral weights each of level
$<2g^*m$.
Then, the following are equivalent:

(a)  $H^0(\Parbun_G,\mathcal{M}(\vec{\lambda},m)^r)\neq 0$, for some integer $r>0$.

(b)  The point $(\frac{\lambda_1^*}{2g^*m}, \dots,  \frac{\lambda_n^*}{2g^*m})$
lies in $\mathscr{C}_{n}$, where $\lambda_k^*:=\kappa(\lambda_k)$ and $\kappa$ is defined in Section \ref{section3}.
\end{theorem}

\begin{proof} We first prove
(a) $\Rightarrow$ (b):

Assume that there is a non-vanishing global section in
 $H^0(\Parbun_G,\mathcal{M}(\vec{\lambda},m)^r)$, non-vanishing at a point $\tilde{\me}=(\mathcal{E}; \bar{g}_1, \dots, \bar{g}_n)$ of $\Parbun_G$.
As in the proof of Lemma \ref{wash}, consider a one parameter family of deformations  $\tilde{\mathcal{E}}_t=(\mathcal{E}_t; \bar{g}_1(t), \dots, \bar{g}_n(t))$ in $\Parbun_G$ parameterized by a smooth curve $S$ so that at a marked point $0\in S$,
 $\tilde{\mathcal{E}}_0= \tilde{\mathcal{E}}$ and the underlying bundle ${\mathcal{E}}_t$ is trivial for general $t\in S$.  Thus, we can assume that
$\mathcal{E}$ is the trivial bundle $\epsilon_G$ and $(\bar{g}_1, \dots, \bar{g}_n)$ are general points of $(G/B)^n$. To prove (b), by \cite[Proposition 4.4]{Teleman-W}, it suffices to show that for any reduction $\tilde{\Pc}=(\Pc; \bar{p_1}, \dots, \bar{p}_n)\in \Parbun_P(d)$ of $(\epsilon_G; \bar{g}_1, \dots, \bar{g}_n)$ to a standard maximal parabolic subgroup $P$ of degree $d$ in relative position $u_1, \dots, u_n\in W^P$, the inequality
\begin{equation} \label{eq22} \sum_{k=1}^n\,\omega_P(u_k^{-1}\lambda_k^*)\leq 2g^*md
\end{equation}
is satisfied.

As in Subsection \ref{section3.6}, induced by the conjugation $\phi_t:P\to P, p\mapsto t^{\bar{x}_P}pt^{-\bar{x}_P}$, $\tilde{\mathcal{P}}$
 admits a one parameter family of deformations  $\tilde{\mathcal{P}}_t=(\Pc_t; \bar{p}_1(t), \dots, \bar{p}_n(t))\in \Parbun_P(d)$ ($t\in \mathbb{A}^1$), such that    $\tilde{\mathcal{P}}_1=
  \tilde{\mathcal{P}}$ and   $\tilde{\mathcal{P}}_0$ comes from the extension of a parabolic $L$-bundle $\tilde{\mathcal{L}}=(\mathcal{L};
\bar{l}_1, \dots, \bar{l_n})\in \Parbun_L(d)$. This gives rise to a morphism
$$\beta: \mathbb{A}^1\to \Parbun_G,\,\,\,\, t\mapsto \tilde{\mathcal{E}}_t=(\mathcal{E}_t:\mathcal{P}_t\times^P G; \bar{g}_1(t), \dots, \bar{g}_n(t)),$$
where $\bar{g}_k(t):=\bar{p}_k(t)u_k^{-1}B$. Pulling back the line bundle $\mathcal{M}(\vec{\lambda},m)^r$ via $\beta$, we get a
$\mathbb{G}_m$-equivariant line bundle (denoted) $\mathcal{M}$ over $\mathbb{A}^1$ with a nonzero section, where $\mathbb{G}_m$
acts on $\mathbb{A}^1$ via multiplication. We now calculate the action of $\mathbb{G}_m$ on the fiber of $\mathcal{M}$ over $0$:

We first calculate the $\mathbb{G}_m$-action on $D(\mathcal{L}\times^L\frg)$. Decompose $\frg$ into eigenspaces $\frg_{\gamma}$ such that the $L$-central
one parameter subgroup  $t^{\bar{x}_P}$ acts on
$\frg_{\gamma}$ by multiplication by $t^{-\gamma}$. Then the desired action is by $t$ raised to the exponent
\begin{align} \label{eq23} \sum_\gamma (\dim \frg_{\gamma}+\deg (\mathcal{L}\times^L \frg_{\gamma}))\gamma&=
\sum_\gamma \deg (\mathcal{L}\times^L \frg_{\gamma})\gamma,\,\,\,\text{since}\,\,\dim \frg_{\gamma}=\dim \frg_{-\gamma}\notag\\
&=d\sum_{\alpha\in R}\alpha(\bar{x}_P)\alpha(\alpha_P^\vee),\,\,\,\,\text{from the definition of}\,\,d\notag\\
&=2g^*d \langle \bar{x}_P, \alpha_P^\vee\rangle.
\end{align}
The last equality follows from Lemma \ref{coxeter}.

It is easy to see that the action of $t^{\bar{x}_P}$ on the fiber $\mathcal{L}_k(\lambda_k)_{\tilde{\mathcal{L}}}$ of  $\mathcal{L}_k(\lambda_k)$ over $\tilde{\mathcal{L}}$ is given by the exponent $-(u_k^{-1}\lambda_k)(\bar{x}_P)$. Thus, combining the equation \eqref{eq23} with the above expression of the action of $t^{\bar{y}_P}$ on the fiber $\mathcal{L}_k(\lambda_k)_{\tilde{\mathcal{L}}}$, we get that the action of $t^{\bar{x}_P}$ on the fiber of $\mathcal{M}$ over $0$ is given by the exponent
\begin{equation}\label{eq23'} \mu:=2g^*md \langle \bar{x}_P, \alpha_P^\vee\rangle -\sum_{k=1}^n\,(u_k^{-1}\lambda_k)(\bar{x}_P).
\end{equation}
But, by Lemma \ref{oldnew}, $\mu \geq 0.$
From this and the equation \eqref{eq20}, we get the equation \eqref{eq22}. This proves the (b)-part of the proposition. 

Proof of `(b) $\Rightarrow$ (a)': By  assumption, each $\lambda_k^*/2g^*m$ is a rational, regular element of $\mathscr{A}$ with
$\theta(\lambda_k^*/2g^*m)<1$. Let $\vec{\mu}=(\lambda_1^*/2g^*m, \dots, \lambda_n^*/2g^*m)$. Assign these weights to the parabolic points of $\pone$. By the assumption (b), the semistable parabolic moduli space of parabolic $G$-bundles on $\pone$ (corresponding to the weights $\vec{\mu}$) is non-empty.
Let $\tilde{\me}$ be a semistable parabolic bundle on $\pone$ with the given weights. Under the ramified cover correspondence as in Theorem
\ref{equivalence},   $\tilde{\me}$ corresponds to a semistable ($\Gamma$-equivariant) principal $G$-bundle $\me'$ on a (suitable) ramified cover $C$ of $\pone$ under the map
$f:\Parbun_G\to \Bun_G(C)$ (which is the inverse of the isomorphism of Theorem \ref{equivalence}), where $\Bun_G(C)$ is the moduli stack of
(non-parabolic) principal $G$-bundles on $C$. Let $\mathbb{D}_C$ be the line bundle over
$\Bun_G(C)$ which assigns to $\mathcal{F}\in\Bun_G(C)$ the line $D(C,\mathcal{F}\times^G \frg)$. We have the pull-back map
$$H^0(\Bun_G(C), \mathbb{D}_C^r)\to H^0(\Parbun_G,f^*\mathbb{D}_C^r).$$
Since our $\me'$ is semistable, there is
a global section $s\in H^0(\Bun_G(C), \mathbb{D}_C^r)$ for a suitable $r>0$ which is non-vanishing at $\me'$:

 This follows from the isomorphism (cf. \cite[$\S$9.3]{LS}; by taking $N$ large enough we can insure that $C$ has genus $\geq 2$)
$$H^0(\Bun_G(C), \mathbb{D}_C^r)\simeq H^0(\mathfrak{M}_G(C), \Theta_{ad}^r),$$
 where
$\Theta_{ad}$ is the theta bundle on the moduli {\it space} $\mathfrak{M}_G(C)$ of semistable principal $G$-bundles on $C$ associate to the adjoint representation of
$G$.
Further, since $\mathfrak{M}_G(C)$ is a projective variety and the Picard group Pic ($\mathfrak{M}_G(C)$) $\simeq \mathbb{Z}$ (cf. \cite[Theorem 2.4]{KN}), we get that $\Theta_{ad}$
is ample and hence there exists $s'\in H^0(\mathfrak{M}_G(C), \Theta_{ad}^r),$ for large enough $r$, such that $s'(\me')\neq 0$.

Therefore, the pull-back of this section to $\Parbun_G$ is non-vanishing at $\tilde{\me}$. So, to finish the proof of (a), we need to know that $f^*\mathbb{D}_C^m$ equals some power of $\mathcal{M}(\vec{\lambda},m)$, which follows from Proposition \ref{pullback} (replacing $r$ by $rm$).
\end{proof}
 \begin{remark} \label{remark5.3} (1) By \cite[$\S$8.9]{LS}, the condition (a) in the above theorem is equivalent to the non-vanishing of the space of conformal blocks
on $\pone$ associated to the weights $r\vec{\lambda}$ at level $2rg^*m$.

(2) For the implication `(a) $\Rightarrow$ (b)' in the above proposition, we do not need to assume that each $\lambda_k$ is regular in our above proof (since we verified all necessary inequalities).  Also, the
implication `(b) $\Rightarrow$ (a)' is  true without this restriction using \cite{BS} as we explain below. The map
$\tilde{f}:\Parbun_G\to\Bun_G^\Gamma(C)$ exists (as in Section \ref{modification} below) even without the assumption of regularity (and the level is allowed to equal $2g^*m$).
It follows from \cite{BS} that $\tilde{f}$ is surjective on objects \cite[Theorem 4.1.5]{BS}. Therefore it remains to show that there is a point of $\Bun_G^\Gamma(C)$ which maps to a semistable point of $\Bun_G(C)$. We can construct this point by considering the principal $G$-bundle corresponding to the local system (of $K$-spaces) in (b) pulled up to $C$ and extended over the punctures (where the local monodromies are trivial). Since this $\Gamma$-equivariant bundle comes from a representation of the fundamental group of $C$ in $K$, it is semistable by Ramanathan's generalization of the  Narasimhan-Seshadri theorem.
\end{remark}

 \section{Comparison of determinants of cohomology under an elementary modification}\label{modification}
 Let $G$ be a connected reductive group in this section. The calculation will be applied to the simple groups of the earlier sections as well as their Levi subgroups. We analyze the effect of an ``elementary modification" on determinants of cohomology.

Let $\me$ be a principal $G$-bundle on a smooth irreducible  projective curve $C$. Let $0\in C$ with formal parameter $z$. Fix a section $s\in \me(D)$, i.e., a trivialization
 $s$ of $\me$ over the formal disc $D$ at $0$.  For $h\in \frh$ such that $\exp(2\pi i h)=1$, consider the map $\ell:D\to H$ given by $\ell(z)=z^h$,
where $\frh$ is the Lie algebra of a maximal torus $H$ of $G$.

 Consider a new principal $G$-bundle $\me'=\me_\ell$ which coincides with $\me$ outside of $0$. Sections of $\me'$ over $D$ are meromorphic sections $s a(z)$
 of $\me$ over $D$ such that $\ell(z)a(z)$ is regular at $0$. We have a section ${s}'=s\ell (z)^{-1}$ of $\me'$ over $D$. Now, consider a representation $G\to \operatorname{GL}(V)$. Decompose $V$ into eigenspaces under the action of $H$:
$$V=\oplus_{\gamma\in A\subset X(H)} V_\gamma,$$
where $A\subset X(H)$ is defined to be the subset such that $V_\gamma\neq 0$.
 We assume that $A$ is symmetric under taking negatives.

We want to compare the determinants of cohomologies of $\me\times^G V$ and $\me'\times^G V$. With our
data, over $D$, we can write  as a trivial vector bundle:
\begin{equation}\label{eqn6.0} \me\times^G V=\oplus_{\gamma} (\mathcal{O} (D)s)\otimes V_{\gamma}=\oplus_{\gamma} L_{\gamma}, \,\,\,
\text{and}\,\, \me'\times^G V=\oplus_{\gamma} (z^{-h}\mathcal{O} (D)s)\otimes V_{\gamma}=\oplus_{\gamma}L'_{\gamma},
\end{equation}
under the meromorphic identification of $L_{\gamma}$ with $L'_{\gamma}$ over $D$.

\begin{defi} \label{ratio} Let $\mv$ and $\mv'$ be vector bundles over $C$ which are identified outside of $0$. Then, as in \cite{BBE},
there is a well defined
line $[\mv:\mv']$ (which formally stands for $\det(\mv/\mv')$) and a canonical isomorphism
\begin{equation}\label{eqn6.1} D(\mv')=D(\mv)\tensor [\mv:\mv'].
\end{equation}

To define it,  find a large positive $k$ so that
$\mv(k)\supset\mv'$ and set
$$[\mv:\mv']= \det(\mv(k)/\mv)^{-1}\tensor \det(\mv(k)/\mv').$$
Note that  $[\mv:\mv']$ is multiplicative.
\end{defi}

Applying this to $\me\times^G V$ and $\me'\times^G V$, we get
\begin{equation} \label{eqn6.1'} D(\me'\times^G V)\tensor D(\me\times^G V)^{-1}=\tensor_{\gamma\in A}[L_{\gamma}:L'_{\gamma}].
\end{equation}
The following lemma follows easily from the equation \eqref{eqn6.0} and the definition of $[- : -].$
\begin{lemma} For any $\gamma\in A$ with $\gamma (h)<0$,
$$[L_{\gamma}:L'_{\gamma}]=\wedge^{\text{top}}\bigl(\frac{\mathcal{O}(D)s}{z^{-\gamma(h)}\mathcal{O}(D)s}\otimes V_\gamma\bigr),$$
and, for any $\gamma\in A$ with $\gamma (h)>0$,
$$[L_{\gamma}:L'_{\gamma}]=\Bigl(\wedge^{\text{top}}\bigl(\frac{z^{-\gamma(h)}\mathcal{O}(D)s}{\mathcal{O}(D)s}\otimes V_\gamma\bigr)\Bigr)^{-1}.$$
\end{lemma}
Suppose $\ell(z)b\ell(z)^{-1}$ is regular at $z=0$ for all $b\in B$, i.e., $h\in \frh_{+}$. Then, changing $s$ to $sb$ does not change
$\me'$. Therefore, we get an action of $B$ (in particular of $H$) on $D(\me'\times^G V)\otimes D(\me\times^G V)^{-1}$. The $H$-action respects
$[L_{\gamma}:L'_{\gamma}]$ for any $\gamma\in A$. Combining the equation \eqref{eqn6.1'} with the above lemma, since $A$ is symmetric, we get the following:

\begin{proposition}\label{ratioweight}
$$D(\me'\times^G V)=D(\me\times^G V)\tensor \bigl(\tensor_{\gamma\in A: \gamma (h)<0} [\wedge^{\text{top}}\bigl(\frac{\mathcal{O}(D)s}{z^{-\gamma(h)}\mathcal{O}(D)s}\otimes V_\gamma\bigr)\tensor \wedge^{\text{top}}\bigl(\frac{z\mathcal{O}(D)s}{z^{-\gamma(h)+1}\mathcal{O}(D)s}\otimes V_\gamma\bigr)]\bigr).$$
Further,
the weight of the $H$-action on $D(\me'\times^G V)\otimes D(\me\times^G V)^{-1}$ is given by
$$-2\sum_{\gamma\in A:\gamma(h)>0} \,\gamma(h)\dim V_\gamma\cdot \gamma.$$
\end{proposition}
We now come to the proof of the following result, which was used in the proof of  Theorem \ref{evproblem}.
\begin{proposition} \label{pullback} Let the notation and assumptions be as in the proof of  Theorem \ref{evproblem}:`(b) $\Rightarrow$ (a)'. Then,
$$f^*\mathbb{D}_C^m =(\mathcal{M}(\vec{\lambda},m))^N,$$
where $\pi:C \to \pone$ is the map as in Subsection \ref{equivariant} of degree $N$.
\end{proposition}
\begin{proof} Take $\tilde{\mathcal{F}}=(\mathcal{F}; \bar{g}_1, \dots, \bar{g}_n) \in \Parbun_G$ and let  $\me:=\pi^*\mathcal{F}$. By the projection formula
\begin{equation} \label{eqn6.4'} D(C,\me\times^G \frg)=D(\pone, \mf\times^G \frg)^N.
\end{equation}
 Let $\me'=f(\tilde{\mathcal{F}})$ be the new bundle over $C$ obtained via  modifying the bundle $\me$ at the points $\tilde{b}_{k}$
(by twisting $\me$ at the points $\tilde{b}_k$ via the morphism $\ell_k:\tilde{D}_k\to H, \tilde{z}_k\mapsto \tilde{z}_k^ {N\lambda_k^*/2g^*m}$, where $\tilde{z}_k$ is a local parameter for the formal disk
$\tilde{D}_k$ of $C$ at
 $\tilde{b}_{k}$), where $f$ is the inverse of the ramified cover correspondence
as in Theorem \ref{equivalence} for $\vec{\mu}=(\frac{\lambda_1^*}{2g^*m}, \dots, \frac{\lambda_n^*}{2g^*m})$.   By Definition \ref{ratio},
\begin{equation}\label{eqn6.4} D({\me}'\times^G\frg)= D(\me\times^G \frg)\tensor \bigl(\tensor_{k=1}^n\bigl(\tensor_{\beta\in R\cup\{0\}}[L_\beta(k):L_\beta'(k)]\bigr)\bigr),
\end{equation}
where $L_\beta(k), L'_\beta(k)$ over $\tilde{D}_k$ are defined as in the beginning of this section taking $V$ to be the adjoint representation $\frg$ of $G$.
By Proposition \ref{ratioweight}, the action of $H$ on $\tensor_{\beta\in R\cup\{0\}}[L_\beta(k):L_\beta'(k)]$ is via the linear form on $\frh$:
\begin{align*} \theta_k(h)&=-2\sum_{\beta\in R^+}\,\beta(N\lambda_k^*/2g^*m)\beta(h),\,\,\,\text{for}\,\,h\in \frh\notag\\
&=-\langle \frac{N\lambda^*_k}{m},h\rangle,\,\,\,\text{By Lemma \ref{coxeter}}.\notag
\end{align*}
Thus,
\begin{equation}\label{eqn6.5} \theta_k=-N\lambda_k/m.
\end{equation}
Combining the equations \eqref{eqn6.4'}, \eqref{eqn6.4} and \eqref{eqn6.5}, we get
\begin{align} f^*(\bD_C)^m &=\bD^{Nm}\otimes (\tensor_{k=1}^n \mathcal{L}_k(N\lambda_k)) \notag\\
&=\mathcal{M}(\vec{\lambda},m)^N.
\end{align}
This proves the proposition.

We give another proof of the proposition using the affine flag variety realization of $\Parbun_G$.

\noindent
{\it Second Proof:} Let $\tilde{z}_k$ be a parameter for the curve $C$ at the point $\tilde{b}_k$ and let $z_k$ be a parameter for $\pone$ at the point $b_k$. Let $\mathfrak{G}_{b_k}$ (resp. $\mathfrak{G}_{\tilde{b}_k}$) be the loop group $G(\mathbb{C}((z_k)))$ (resp. $G(\mathbb{C}((\tilde{z}_k))))$ and let  $\mathfrak{B}_{b_k}$ (resp. $\mathfrak{P}_{\tilde{b}_k}$) be the Iwahori (resp. parahoric) subgroup defined as $\{f\in G(\mathbb{C}[[z_k]]):f(0)\in B\}$ (resp. $G(\mathbb{C}[[\tilde{z}_k]])$). Let $\Gamma_{\pone}$ (resp. $\Gamma_C$) be the group of maps $\pone\setminus\{b_1, \dots, b_n\} \to G$
(resp. maps $C\setminus\{\tilde{b}_1, \dots, \tilde{b}_n\} \to G$). Then,
there are canonical identifications (cf. \cite[Theorem 8.5]{LS} for a similar result, the proof of which can be adapted to prove the following; also see \cite{PR}):
\begin{equation} \label{affinestack}\Parbun_G\simeq \Gamma_{\pone}\backslash\bigl(\prod_{k=1}^n\,\mathfrak{G}_{b_k}/\mathfrak{B}_{b_k}\bigr),
\end{equation}
and
\begin{equation} \label{affinestack'} \Bun_G(C)\simeq \Gamma_C\backslash\bigl(\prod_{k=1}^n\,\mathfrak{G}_{\tilde{b}_k}/\mathfrak{P}_{\tilde{b}_k}\bigr),
\end{equation}
where $\Gamma_{\pone}$ acts on $\prod_{k=1}^n\,\mathfrak{G}_{b_k}/\mathfrak{B}_{b_k}$ diagonally via evaluation at $b_k$ on the $k$-th factor
and a similar action of $\Gamma_C$ on $\prod_{k=1}^n\,\mathfrak{G}_{\tilde{b}_k}/\mathfrak{P}_{\tilde{b}_k}$. Now, the map
$f:\Parbun_G\to \Bun_G(C)$ corresponds (under the above identifications) to the map
$$\hat{f}:\Gamma_\pone\backslash\bigl(\prod_{k=1}^n\,\mathfrak{G}_{b_k}/\mathfrak{B}_{b_k} \bigr)\to  \Gamma_C\backslash\bigl(\prod_{k=1}^n\,\mathfrak{G}_{\tilde{b}_k}/\mathfrak{P}_{\tilde{b}_k}\bigr),$$taking
$$\Gamma_{\pone} (g_1(z_1)\mathfrak{B}_{b_1}, \dots, g_n(z_n)\mathfrak{B}_{b_n})
\mapsto  \Gamma_C(g_1(\tilde{z}_1^N)\tilde{z}_1^{-N\mu_1}\mathfrak{P}_{\tilde{b}_1}, \dots, g_n(\tilde{z}_n^N)\tilde{z}_n^{-N\mu_n}\mathfrak{P}_{\tilde{b}_n}).$$
Consider the commutative diagram
\[
\xymatrix{
\prod_{k=1}^n\,\mathfrak{G}_{b_k}/\mathfrak{B}_{b_k}\ar[d]^{{\pi}} \ar[r]^-{\tilde{f}} &\ar[d]^{\tilde{\pi}}\prod_{k=1}^n\,\mathfrak{G}_{\tilde{b}_k}/\mathfrak{P}_{\tilde{b}_k}\\
 \Gamma_{\pone}\backslash\bigl(\prod_{k=1}^n\,\mathfrak{G}_{b_k}/\mathfrak{B}_{b_k}\bigr) \ar[r]^-{\hat{f}} & \Gamma_C\backslash\bigl(\prod_{k=1}^n\,\mathfrak{G}_{\tilde{b}_k}/\mathfrak{P}_{\tilde{b}_k}\bigr),
}
\]
where $\tilde{\pi}$ and $\pi$ are the standard projection maps and $\tilde{f}$ is defined similar to $\hat{f}$.
Now, by \cite[Theorem 6.6 and Lemma 5.2]{KNR}, the bundle $\tilde{\pi}^*(\bD_C^m)$ over $\prod_{k=1}^n\,\mathfrak{G}_{\tilde{b}_k}/\mathfrak{P}_{\tilde{b}_k}$ is the homogeneous line bundle corresponding to the affine characters $(-2g^*m\omega_0, \dots, -2g^*m\omega_0)$ respectively on each factor, where $\omega_0$ is the zeroth affine fundamental weight. Similarly, the bundle $\pi^*(\mathcal{M}(\vec{\lambda},m))$ over $\prod_{k=1}^n\,\mathfrak{G}_{b_k}/\mathfrak{B}_{b_k}$ is the homogeneous line bundle corresponding to the affine characters $(-2g^*m\omega_0-\lambda_1, \dots, -2g^*m\omega_0-\lambda_n)$ respectively on each factor. It is easy to see that the homogeneous line bundle corresponding to the affine characters $(-2g^*m\omega_0, \dots, -2g^*m\omega_0)$ over $\prod_{k=1}^n\,\mathfrak{G}_{\tilde{b}_k}/\mathfrak{P}_{\tilde{b}_k}$ pulls back under $\tilde{f}$ to the homogeneous line bundle corresponding to the affine characters $(-2g^*mN\omega_0-2g^*mN\mu_1^*, \dots, -2g^*mN\omega_0-2g^*mN\mu_n^*)$ over $\prod_{k=1}^n\,\mathfrak{G}_{b_k}/\mathfrak{B}_{b_k}$ (cf. \cite[ Definition 11.3.4]{K} and the fact that
$H^2(\mathfrak{G}_{b_k}/\mathfrak{B}_{b_k}, \mathbb{Z})\simeq\Pic (\mathfrak{G}_{b_k}/\mathfrak{B}_{b_k})$, where the latter can be proved by the same proof as that of \cite[Proposition 2.3]{KNR}). Since $\mu_k=\lambda_k^*/2g^*m$,
and the map $\pi$ induces an injective map of the corresponding Picard groups (cf. \cite[Proof of Proposition 8.7]{LS}), we get the proposition.
\end{proof}

\begin{lemma} \label{cohvanishing}
For any (not necessarily dominant) integral weights
$\vec{\lambda}=(\lambda_1, \dots, \lambda_n)$ and an integer $m$, the space $H^0(\Parbun_G,\mathcal{M}(\vec{\lambda},m))= 0$ unless
$m$ is a non-negative integer and each $\lambda_k$ is of level $2g^*m$.
\end{lemma}
\begin{proof} First of all,  from the affine analogue of the Borel-Weil-Bott theorem (cf. \cite[Corollary 8.3.12]{K}), the space (following the notation in the `Second Proof' of Proposition \ref{pullback})
$H^0\bigl(\mathfrak{G}_{b_k}/\mathfrak{B}_{b_k}, \mathfrak{L}(\lambda_k+2g^*m\omega_0)\bigr)$ is zero unless $m$ is a non-negative integer and
$\lambda_k$ is a dominant integral weight for $G$ of level $2g^*m$, where $\mathfrak{L}(\lambda_k+2g^*m\omega_0)$ is the homogeneous line bundle
over $\mathfrak{G}_{b_k}/\mathfrak{B}_{b_k}$ corresponding to the affine character $\lambda_k+2g^*m\omega_0$. Now, the lemma follows from the identification \eqref{affinestack}, since the line bundle
$\mathcal{M}(\vec{\lambda},m)$ pulls-back (under this identification) to the line bundle
$$\mathfrak{L}(\lambda_1+2g^*m\omega_0)\boxtimes \dots \boxtimes  \mathfrak{L}(\lambda_n+2g^*m\omega_0)\,\,\,\text{over}\,\, \prod_{k=1}^n\,\mathfrak{G}_{b_k}/\mathfrak{B}_{b_k},$$
and the identification of the space of sections $H^0(\Parbun_G,\mathcal{M}(\vec{\lambda},m))$ with
$$H^0\bigl(\prod_{k=1}^n \mathfrak{G}_{b_k}/\mathfrak{B}_{b_k}, \mathfrak{L}(\lambda_1+2g^*m\omega_0)\boxtimes \dots \boxtimes  \mathfrak{L}(\lambda_n+2g^*m\omega_0\bigr)^{\Gamma_{\pone}}$$
(cf. \cite[Section 8.9]{LS}).

We give another proof.

\noindent
{\it Second Proof.} Assume that $H^0(\Parbun_G,\mathcal{M}(\vec{\lambda},m))\neq 0$.
By the method of the proof given in \cite{LS}, it follows immediately that  $\lambda_k$ are dominant. Suppose
$\lambda_1 (\theta^\vee) >k=2g^*m$. Choose $e_\theta\in \frg_\theta$ and  $f_\theta\in \frg_{-\theta}$ such that $[e_\theta, f_\theta]=\theta^\vee$, where $\frg_\theta$ (resp. $\frg_{-\theta}$) is the highest (resp. lowest) root space of $\frg$.  Take distinct points $\{p_1,\dots,p_n\}\subset \Bbb{A}^1 \subset \pone$ so that $p_1=0$. It is  easy to see, using \cite{Beauville,LS}, that $H^0(\Parbun_G,\mathcal{M}(\vec{\lambda},m))$ is identified with the dual of a space of the form
$$\Bbb{V}= \frac{(V_{\lambda_1}\tensor\dots\tensor V_{\lambda_n})}{\frg(V_{\lambda_1}\tensor\dots\tensor V_{\lambda_n}) +\im T^{k+1}},$$
where $k=2g^*m$,  $V_{\lambda_k}$ is  the  finite dimensional irreducible representation of $\frg$ with highest weight $\lambda_k$ and
$T= \sum_{k=2}^n p_k f^{(k)}_{\theta}$ ($f^{(k)}_{\theta}$ denotes the action of $f_\theta$ on the $k$-th factor). We show that $\Bbb{V}=0$,
which will complete the second proof.

We need only show that
$\tau=v_1^+\tensor v_2\tensor \dots\tensor v_n=0\in\Bbb{V}$, where $v^+_1$ is the highest weight vector of $V_{\lambda_1}$ and $v_k\in V_{\lambda_k}$ are arbitrary weight vectors (since these vectors generate $V_{\lambda_1}\tensor\dots\tensor V_{\lambda_n}$ as a  $\frg$-module). To prove this,  introduce the operator
$S=\sum_{k=2}^n\frac{1}{p_k}e^{(k)}_{\theta}$ on $(V_{\lambda_1}\tensor\dots\tensor V_{\lambda_n})$.
Let
$$W_0= (\Bbb{C}v_1^+\tensor V_{\lambda_2}\tensor\dots\tensor V_{\lambda_n})\subseteq (V_{\lambda_1}\tensor\dots\tensor V_{\lambda_n}).$$
Clearly,  $S$ and $T$ act on $W_0$.  Let $\theta^{\vee}(v_k)=\mu_k v_k, k=2,\dots,n$. Then, we
can see that
$[S,T](v_1^+\tensor v_2\tensor\dots \tensor  v _n)= (\sum_{k\geq 2} \mu_k) (v_1^+\tensor v_2\tensor\dots\tensor v_n)$.

This leads to a $\operatorname{sl}_2$-action on $W_0$ with the action of $H\in\operatorname{sl}_2$ by
$$H(v_1^+\tensor v_2\tensor\dots \tensor v_n)=  (\sum_{k\geq 2}\mu_k)  (v_1^+\tensor v_2\tensor\dots\tensor v_n),$$
$E$ by $S$ and $F$ by $T$, where $\{E,F,H\}$ is the standard basis of $\operatorname{sl}_2$.   We can further assume that $\tau$ is of weight $0$,  for otherwise it is already in $\frg(V_{\lambda_1}\tensor\dots\tensor V_{\lambda_n})$ itself. Then, the action of $H$ on $\tau$  is multiplication by
 $\sum_{k\geq 2} \mu_k=-\lambda_1(\theta^{\vee})< -k$. Now, note that in any (not necessarily irreducible) finite dimensional $\operatorname{sl}_2$-representation
$W$,  if  $H$ acts on $\tau\in W$ by $-m$ with $m>0$, then $\tau$ can be written as $F^m  \tau'$, for some $\tau' \in W$. Applying
 this to our $\tau$, we see that  $\tau=T^{k+1}\tau'$ for some $\tau'\in W_0$. Hence, $\tau=0$ in $\Bbb{V}$.
\end{proof}
\section{Levi twistings}

Let $G$ be a  simple, simply-connected complex algebraic group and let $P$ be a  standard  maximal parabolic subgroup. Let $L$ be the Levi subgroup of $P$ containing the maximal torus $H$.
\subsection{The structure of Levi subgroups} \label{levisub} Note that
$$\pi_2(G/P)\simeq H_2(G/P)=\Bbb{Z}\,\,\,\text{ and} \,\,\pi_2(G)=0,$$ where the first identification is via the Hurewicz theorem
and the second identification with $\Bbb{Z}$ is by the Bruhat decomposition.  (Of course $\pi_2(G)=0$ for any Lie group $G$.)
 Hence, from the long exact homotopy sequence for the fibration $G \to G/P$, we get the isomorphism:
\begin{equation}\label{pione}
\beta:\Bbb{Z}=H_2(G/P)\to \pi_1(P)\simeq \pi_1(L),
\end{equation}
where the last isomorphism of course follows  since $L$ is a deformation retract of $P$.

The fundamental character $\omega_P$ extends to a character (still) denoted by  $\omega_P:L\to \Bbb{G}_m$ inducing a map $\pi_1(L)\to \pi_1(\Bbb{G}_m)=\Bbb{Z}$. The loop $z^{\alpha_P^\vee}$ goes over to $1\in\Bbb{Z}$.  Therefore, $\omega_P:L\to \Bbb{G}_m$ induces an isomorphism on fundamental groups.

Consider the maximal semisimple subgroup $L'=[L,L]$. Then, $L'$ is a (connected) simply-connected group. To prove this, use the long exact
homotopy sequence corresponding to the fibration $\pi:L\to L/L'$ together with the isomorphism \eqref{pione} and the fact that $ L/L'$
is a one dimensional torus. Let $Z^o$ be the identity component of the center $Z$ of $L$. Then, the canonical map $i:Z^o\to L/L'$ is clearly an isogeny.

Consider the pull-back diagram:
\[
\xymatrix{
\tilde{L}\ar[d]^{\tilde{\pi}} \ar[r]^-{\tilde{i}} & L\ar[d]^{{\pi}}\\
Z^o  \ar[r]^-{{i}} & L/L'.
}
\]
Then, the central inclusion $Z^o\hookrightarrow L$ splits the left vertical projection, giving rise to a canonical isomorphism:
$$\tilde{L}\simeq L'\times Z^o.$$
As in equation \eqref{eqn0'}, let $N_P$ be the smallest positive integer such that $\bar{x}_P:=N_Px_P$ belongs to the coroot lattice of $G$. Then, $z\mapsto z^{\bar{x}_P}$ gives an isomorphism $\mathbb{G}_m\to Z^o.$ Let $k_L$ be the order of the cokernel of  $ \pi_1(Z^o)\to \pi_1(L)$. Then, $k_L$ is also the order of
$\Ker i$ (and  $\Ker \tilde{i}$).

\subsection{Degrees of principal $L$-bundles}
Recall that the degree of a principal $L$-bundle $\ml$ over $\pone$ is the first Chern class of the line bundle  $\mathcal{L}\times^L\Bbb{C}_{-\omega_P}\to \pone$.  Since topologically the principal $L$-bundles on $\pone$ are classified by $\pi_1(L)$ via
the clutching construction, one can view
the degree as an element in $\pi_1(L)\simeq \bz$ (under the canonical identification as in \eqref{pione}). We will be interested in considering  the $k_L$-degree
$$\deg_{k_L}(\ml):=\deg (\ml) \,\,\,(\text{mod}\,\, k_L)$$
of a principal $L$-bundle.

\subsection{Moduli stacks}
Consider the stack $\Parbun_L$ parameterizing the quasi-parabolic $L$-bundles $\tilde{\mathcal{L}}=(\ml; \bar{l}_1, \dots, \bar{l}_n)$ on $\pone$ consisting of a principal $L$-bundle $\ml$ on $\pone$ with parabolic
 structure $\bar{l}_k\in \mathcal{L}_{b_k}/B_L$, $k=1,\dots,n$. This breaks up according to the degrees $d\in\pi_1(L)$:
$$\Parbun_L =\bigsqcup_d \Parbun_L(d).$$

\subsection{Line bundles on $\Parbun_L$}
Given the characters $\vec{\lambda}=(\lambda_1,\dots,\lambda_n)$ of $H$ and an integer $m$, analogous to the definition of
$\mathcal{M}(\vec{\lambda},m)$ as
in Section \ref{ev}, we can form the  line bundle $\mathcal{N}(\vec{\lambda},m)$
on $\Parbun_L$ whose fiber over $\tilde{\ml}=(\mathcal{L}, \bar{l}_1,\dots,\bar{l}_n)$ is the line
$$D(\mathcal{L}\times^L\frg)^{m}\tensor (\tensor_{k=1}^n \bar{\mathcal{L}}_k(\lambda_k)),$$
where $\bar{\mathcal{L}}_k(\lambda_k)$ is the fiber of the line bundle  $\mathcal{L}_{b_k}\times^{B_L}\mathbb{C}_{-\lambda_k}\to  \mathcal{L}_{b_k}/B_L$
over $\bar{l}_k$.
\subsection{Sections of $\mathcal{N}(\vec{\lambda},m)$  over $\Parbun_L(d)$} \label{centralaction}
By a computation analogous to the equation \eqref{eq23'}, we see that  the identity component $Z^o$ of the center of $L$ acts trivially on
$\mathcal{N}(\vec{\lambda},m)$ if and only if
\begin{equation}\label{scale}
\sum_{k=1}^n \lambda_k(x_P)=4g^*dm/\langle \alpha_P, \alpha_P\rangle.
\end{equation}
\subsection{Existence of global sections of line bundles over $\Parbun_L(d)$} Let $\fl'$ be the Lie algebra of $L'$ and let $\fh':=\fh\cap \fl'$ be its Cartan subalgebra.
We can not apply Theorem  \ref{evproblem} directly since $L$ is not semisimple. We first consider the case when $k_L$ divides $d$. Write $d=k_Ld'$. Now, consider the map
\begin{equation}\label{stackmap}
\eta: \Parbun_{L'}\to \Parbun_L(d),
\end{equation}
which sends a parabolic $L'$-bundle
$\tilde{\ml}'$ to the parabolic $L$-bundle $\tilde{\ml}$ obtained as the image of the principal $\tilde{L}$-bundle ${\ml}'\times A$
under $\tilde{i}$, where $A$ is the principal $\Bbb{G}_m$-bundle corresponding to $\mathcal{O}(d')$. (The parabolic structure on $\ml$ is the canonical one coming from the parabolic structure on $\ml'$.)

This gives rise to the pull-back map
 \begin{equation}\label{themap}
 \eta^*:H^0(\Parbun_L(d),  \mathcal{N}(\vec{\lambda}, m))  \to  H^0(\Parbun_{L'},\mathcal{N}'(\vec{\lambda}', m)),
 \end{equation}
where $\vec{\lambda}':=({\lambda'_1}, \dots, {\lambda'_n}), \lambda_k':={\lambda_k}_{|\fh'}$ and $\mathcal{N}'(\vec{\lambda}', m)$ is the line bundle which assigns to a parabolic $L'$-bundle $\tilde{\mcl}'$, the line
$$D({\ml}'\times^{L'}\fl')^{m}\tensor (\tensor_{k=1}^n \bar{\mathcal{L}}_k(\lambda'_k)).$$
\begin{lemma}\label{levitosemisimple}
The map $\eta^*$  is an isomorphism for any $\mathcal{N}(\vec{\lambda}, m)$ such that the equation \eqref{scale} is satisfied.
\end{lemma}
\begin{proof}
The map $\eta$ of stacks  is surjective since the structure group of any principal $L$-bundle over $\pone$ of degree divisible by $k_L$ lifts to $\tilde{L}$ and, moreover, $L'/B_{L'}= L/B_L$. Hence, $\eta^*$  is injective.

To show that $\eta^*$  is surjective, pick a
$\phi'\in  H^0(\Parbun_{L'},\mathcal{N}'(\vec{\lambda}', m))$. We need  to lift the section $\phi'$ to a section
$\phi \in H^0(\Parbun_L(d),  \mathcal{N}(\vec{\lambda}, m))$.  Since  $L'/B_{L'}= L/B_L$, it suffices to show that for any $\tilde{\mcl}_1,
\tilde{\mcl}_2\in \Parbun_{L}(d)$,  lifts $\tilde{\mcl}'_1,
\tilde{\mcl}'_2\in \Parbun_{L'}$ (via $\eta$), and any $L$-bundle isomorphism $\phi:{\mcl}_1 \to
{\mcl}_2$, there exists $z\in Z^o$ such that the composed $L$-bundle isomorphism  $m_z\circ\phi:{\mcl}_1 \to
{\mcl}_2$ (where $m_z:{\mcl}_2 \to
{\mcl}_2$ is the isomorphism $e\mapsto ez^{-1}$) lifts to a $L'$-bundle isomorphism $\phi':{\mcl}'_1 \to
{\mcl}'_2$ (since, by  assumption,  $Z^o$ acts trivially on  $\mathcal{N}(\vec{\lambda}, m)$).
To prove this, take any elements $e_1(x)\in {\mcl}'_1(x), e_2(x)\in {\mcl}'_2(x)$ in the fiber over $x\in \pone$. Then, we can write (thinking of
$e_1(x)$ and $e_2(x)$ as elements of $\mcl_1(x)$ and $\mcl_2(x)$ respectively)
$$\phi(e_1(x))=e_2(x)l(x),\,\,\,\text{for some} \,\,l(x)\in L.$$
It is easy to see that $l(x) L'\in L/L'$ does not depend upon the choices of $e_1(x)$ and $e_2(x)$.  In particular, we get a function
$l:\pone \to L/L'$, which must be constant say $l_oL'$. Moreover, since $Z^o$ maps surjectively onto $L/L'$ (see Subsection \ref{levisub}), we can take
$l_o\in Z^o$. This provides the lifting of $m_{l_o}\circ\phi$ to
a $L'$-bundle isomorphism $\phi':{\mcl}'_1 \to
{\mcl}'_2$.
\end{proof}

\subsection{Changing the $k_L$-degree} Let $G,P$ be as in the beginning of this section and let $Q^\vee\subset \frh$ be the coroot lattice of $G$.
\begin{lemma} \label{specialmu} There exists an element $\mu_P\in Q^\vee$satisfying the following:

(a) $0\leq \alpha (\mu_P)\leq 1$, for all the roots $\alpha\in R^+_\fl$, where $ R^+_\fl$ is the set of positive roots of $\fl$, and

(b) $|\omega_P(\mu_P)|=1$.
\end{lemma}
\begin{proof} If the maximal parabolic $P$ is such that $\alpha_P$ is a long root (for simply-laced groups all the roots are considered long), then take $\mu_P=-\alpha_P^\vee$. By \cite[Page 278]{B}, since
$|\alpha (\mu_P)|\leq 1$, for all the roots $\alpha \neq \pm\alpha_P$, we get that (a) is satisfied for $\mu_P=-\alpha_P^\vee$
(since $\langle \alpha_i, \alpha_P^\vee\rangle \leq 0$ for any $\alpha_i\neq \alpha_P$). Of course, (b) is satisfied for $\mu_P=-\alpha_P^\vee$ (with $\omega_P(\mu_P)=-1$).

Following the Bourbaki \cite[Planche I - IX]{B} convention, this leaves us with the following cases to consider, where we give an explicit $\mu_P$ in each case. (In the following, we denote by $P_i$ the maximal parabolic subgroup with $S_{P_i}=\{\alpha_i\}$ and $\theta$ denotes the highest root of $G$.)

(1) $G=B_\ell, P=P_\ell :$ Take $\mu_{P_\ell}=\theta^\vee$,

(2) $G=C_\ell, P=P_i \,( 1\leq i < \ell) :$ Take $\mu_{P_i}=\theta^\vee$,

(3) $G=G_2, P=P_1:$ Take $\mu_{P_1}=\theta^\vee$,

(4) $G=F_4, P=P_3:$ Take $\mu_{P_3}=\epsilon_1-\epsilon_4$ , and

(5) $G=F_4, P=P_4:$ Take $\mu_{P_4}=\theta^\vee$ .
\end{proof}

We refer to  \cite{Ram} for an analogue of the following construction.
Choose a point $c\in \pone$ distinct from $b_1,\dots, b_n$. For the convenience of the notation take $c=0$. Let $\Parbun^{[1]}_L(d)$ be the moduli stack of parabolic $L$-bundles of degree $d$ with
parabolic structures at $b_1,\dots, b_n, c$. There is the forgetful morphism $\Parbun^{[1]}_L(d)\to\Parbun_L(d)$ (forgetting the parabolic structure at $c$). It is clear that
pulling back $\mathcal{N}(\vec{\lambda},m)$ results in a similar line bundle (with $\lambda_{n+1}=0$ and no changes in $\lambda_1,\dots,\lambda_n$) without any change in the space of global sections.

We describe an operation which results in a morphism $\Parbun^{[1]}_L(d)\to \Parbun^{[1]}_L(d\pm 1)$. Fix any $\mu\in Q^\vee$ and let $\ell_\mu:\mathbb{C}^* \to H\subset L$ be the one parameter subgroup $\ell_\mu(z)=z^\mu$.
The operation is as follows: Fix a trivialization $s$ of  $\tilde{\mm}=(\mm; \bar{l}_1, \dots, \bar{l}_{n+1})\in \Parbun^{[1]}_L(d)$ in a formal
neighborhood $D$ of $c=0$ so that $s(0)B_L=\bar{l}_{n+1}$. The new principal $L$-bundle $\mm_\mu$, which coincides with $\mm$ outside of $0$, is defined as follows: Sections of $\mm_\mu$ over $D$ are meromorphic sections $s a(z)$
 of $\mm$ so that $\ell_\mu(z)a(z)$ is regular at $0$. We have a section
\begin{equation} \label{eqn7.0}  {s}'=s\ell_\mu (z)^{-1}\,\,\,\text{ of }\,\, \mm_\mu\,\,\,\text{over} \,\, D.
\end{equation}
The construction of the new bundle $\mm_\mu$ depends on the choice of the trivialization $s$. In particular, it may not give a well defined
bundle $\mm_\mu$ for an arbitrary choice of $\mu$. However, we have the following  result:
\begin{lemma} Assume that $\mu\in Q^\vee$ satisfies the condition (a)  of Lemma \ref{specialmu}. Then, for any $\tilde{\mm}=(\mm; \bar{l}_1, \dots, \bar{l}_{n+1})\in \Parbun^{[1]}_L(d)$, the above operation gives a well-defined bundle  $\mm_{\mu}$, i.e., the construction of $\mm_\mu$ does not depend on the choice of the trivialization $s$ satisfying $s(0)B_L=\bar{l}_{n+1}$.

Further, the degree $d'$ of  $\mm_\mu$ is $d+\omega_P(\mu)$.
\end{lemma}
\begin{proof}
We first show that changing $s(z)$ by $s(z) c(z)$ with $c(z)\in L[[z]]:=L(\mathbb{C}[[z]])$ and $c(0)\in B_L$ does not change $\mm_\mu$ (in its meromorphic identification with $\mm$). To prove this, it suffices to show that
 $\ell_\mu(z)c(z)\ell_\mu(z)^{-1}$ is regular at $z=0$:

Considering the embedding  $L[[z]]\subset G[[Z]]$, and the affine Kac-Moody Lie algebra $\hat{\ml}(\fg):=\mathbb{C}[z,z^{-1}]\otimes_\mathbb{C} \fg\oplus \mathbb{C} c\oplus \mathbb{C} d$ (cf., \cite[Section 13.1]{K}), it suffices to show that
\begin{equation}\label{eqn7.1} \Ad\ell_\mu(z)\cdot (z\mathbb{C}[z]\otimes \fg_\beta) \subset \mathbb{C}[z]\otimes \fl,\,\,\,\text{for any}\,\,
\beta\in R_\fl,
\end{equation} and
\begin{equation}\label{eqn7.2} \Ad\ell_\mu(z)\cdot  \fg_\beta \subset \mathbb{C}[z]\otimes \fl,\,\,\,\text{for any}\,\,
\beta\in R^+_\fl,
\end{equation}
where $\fg_\beta$ is the root space of $\fg$ corresponding to the root $\beta$.
But, as it is easy to see, for any $\mu\in Q^\vee$ and any $\beta\in R$,
\begin{equation}\label{eqn7.3} \Ad (z^\mu)\cdot (z^n\otimes \fg_\beta) \subset z^{n+\beta(\mu)}\otimes \fg_\beta.
\end{equation}
From this equation, the equations \eqref{eqn7.1} and \eqref{eqn7.2} follow for any $\mu$ satisfying the condition (a) of Lemma
\ref{specialmu}.

By an easy calculation,
 the degree $d'$ of $\mm_{\mu}$ is $d+\omega_P(\mu)$.
 \end{proof}

We now describe the parabolic structure on $\mm_\mu$ for any $\mu$ as in the above lemma. For any $k=1, \dots , n$, the parabolic structure on
 $\mm_\mu$ is the same as that of $\tilde{\mm}$ at $b_k$. To describe the parabolic structure at $c$,  following \cite{B2}, consider the  subgroup $E_\mu$ of $L$ as the set of all limits
  $$\lim_{z\to 0} \ell_\mu(z)c(z)\ell_\mu(z)^{-1},\  c(z)\in L[[z]]\,\,\,\text{with}\,\,\ c(0)\in B_L.$$
Then, by the equation \eqref{eqn7.3}, it is easy to see that the Lie algebra
\begin{equation}\label{eqn7.3'}\text{Lie}(E_\mu)=\fh\oplus \bigl(\oplus_{\beta\in R_\fl^+:\beta(\mu)=0}\,\fg_\beta\bigr)\oplus \bigl(\oplus_{\beta\in R_\fl^+:\beta(\mu)=1}\,\fg_{-\beta}\bigr).\end{equation}
In particular, $E_\mu$ is a Borel subgroup of $L$ containing $H$. Hence, there exists a unique Weyl group element $w_\mu$ of $L$ such that
$E_\mu=w_\mu B_Lw_\mu^{-1}$.
 Thus, we get a well defined point in $({\mm_\mu})_0/w_\mu B_L w_\mu^{-1}$ by taking $s'(0)$ (mod $w_\mu B_L w_\mu^{-1}$), for any section $s'$ of $\mm_\mu$ over $D$ defined by
the equation \eqref{eqn7.0}. By the definition of $E_\mu$, it is easy to see that  the element $s'(0)$ (mod $w_\mu B_L w_\mu^{-1}$) does not
depend upon the choice of the section $s$ of $\mm$ over $D$ satisfying $s(0)B_L=\bar{l}_{n+1}$.

For any principal $L$-bundle $\mathcal{E}$ over a scheme $Y$, any $y\in Y$, and $w$ in the Weyl group of $L$, we can identify
$$\theta_w: \mathcal{E}_y/(wB_L w^{-1})\to \mathcal{E}_y/B_L$$ induced by  $e\mapsto ew$. This allows us to define an element
$\bar{l}_{n+1}^\mu$ in the fiber $(\mm_\mu)_{c}/B_L$ as the image of $s'(0)$ (mod $w_\mu B_L w_\mu^{-1}$) under $\theta_{w_\mu}$. So, we get the parabolic bundle
$\tilde{\mm}_\mu:=(\mm_\mu; \bar{l}_1, \dots, \bar{l}_n,\bar{l}_{n+1}^\mu)\in \Parbun^{[1]}_L(d)$.

 All in all, we therefore have completed our description of $\tau_\mu:\Parbun^{[1]}_L(d)\to \Parbun^{[1]}_L(d+\omega_P(\mu))$. Taking $\mu=\mu_P$ as in Lemma \ref{specialmu}, we get a map $\tau_{\mu_P}:\Parbun^{[1]}_L(d)\to \Parbun^{[1]}_L(d\pm 1)$.

\begin{lemma} \label{iso} The map $\tau_\mu: \Parbun^{[1]}_L(d)\to \Parbun^{[1]}_L(d+\omega_P(\mu))$, for any $\mu$ satisfying the condition (a) of Lemma \ref{specialmu}, is an isomorphism.
\end{lemma}
\begin{proof} Let $w_\mu\in W_L$ be the element such that $E_\mu=w_\mu B_Lw_\mu^{-1}$, where Lie ($E_\mu$) is defined by the equation
\eqref{eqn7.3'}. Then, it is easy to see that  $\tau_{-w^{-1}\mu}: \Parbun^{[1]}_L(d+\omega_P(\mu))\to \Parbun^{[1]}_L(d)$ is a well defined morphism of stacks. Moreover, it is the inverse of $\tau_\mu$, proving that $\tau_\mu$ is an isomorphism. (Observe that $w_{-w^{-1}\mu}=w_\mu^{-1}$.)
\end{proof}

 Finally, we need to consider the pull-back of line bundles under the above map $\tau_\mu$.
\begin{lemma} \label{bundleiso} For any $\mu$ satisfying the condition (a) of Lemma \ref{specialmu}, any characters $\vec{\lambda}=(\lambda_1, \dots, \lambda_n)$ of $H$ and any integer $m$,
$$\tau_\mu^*\bigl(\mathcal{N}((\vec{\lambda}, 2g^*m\mu^*),m)\bigr)=\mathcal{N}((\vec{\lambda}, 0),m),$$
where $\mu^*=\kappa^{-1}(\mu)$  ($\kappa$ being defined in Section \ref{section3}).
\end{lemma}
\begin{proof} By Proposition \ref{ratioweight}, for any $\tilde{\mm}\in \Parbun^{[1]}_L(d)$,
$D(\mm'\times^L\frg)\otimes D(\mm\times^L\frg)^{-1}$ is the line $\mm_0\times^L\mathbb{C}_\delta$, where
 $\mm'$ is the underlying bundle of $\tau_\mu(\tilde{\mm})$ and $\delta \in \frh^*$ is given by (for any $x\in \frh$)
\begin{align*} \delta(x)&=-2\sum_{\gamma\in R^+}\,\gamma(\mu) \gamma(x),\\
&=-2g^*\langle\mu, x\rangle,\,\,\,\text{by Lemma}\, \ref{coxeter},
\end{align*}
where $R^+$ is the set of positive roots of $\frg$. Thus, $\delta =-2g^*\mu^*$, where $\mu^*=\kappa^{-1}(\mu)$. Thus, the pull-back of $\mathcal{N}((\vec{\lambda}, 2g^*m\mu^*),m)$ under $\tau_\mu$ is given by $\mathcal{N}((\vec{\lambda}, 0),m)$.
This proves the lemma.
\end{proof}

Let $\mu_P\in Q^\vee$ be any element satisfying Lemma \ref{specialmu} and let $d_o$ be the smallest positive integer such that
$d+d_o\omega_P(\mu_P)\equiv 0$ (mod $k_L$), where $k_L$ is defined in Subsection \ref{levisub}. Choose points $c_1, \dots, c_{d_o} \in \pone$ distinct
from the points $b_1, \dots, b_n$. For any $0\leq r\leq d_o$, let $\Parbun_L^{[r]} (d)$ be the moduli stack of quasi-parabolic principal $L$-bundles on $\pone$ of degree $d$ with parabolic structures at  $b_1, \dots, b_n, c_1, \dots, c_r$. Thus, $\Parbun_L^{[0]} (d)=\Parbun_L (d)$. There is a
similar definition of $\Parbun_{L'}^{[r]}$, where we take $n+r$ parabolic points.

Combining Lemmas \ref{iso} and \ref{bundleiso}, we get the following result.
\begin{corollary} \label{cor7.6} For  any $\mu_P$ satisfying Lemma \ref{specialmu}, any characters $\vec{\lambda}=(\lambda_1, \dots, \lambda_n)$ of $H$ and any integer $m\geq 0$ such that the equation \eqref{scale} is satisfied, we have
\begin{equation}\label{eigeniso}
 H^0(\Parbun_L(d),  \mathcal{N}(\vec{\lambda}, m))  \simeq  H^0(\Parbun_{L'}^{[d_o]},\mathcal{N}'((\vec{\lambda}', [2g^*m(\mu_P^*)_{|\frh'}]^{d_o}), m)),
 \end{equation}
where $\vec{\lambda}':=({\lambda'_1}, \dots, {\lambda'_n}), \lambda_k':={\lambda_k}_{|\fh'}$ and $[\lambda]^{d_o}$ denotes $d_o$ copies
of $\lambda$.
\end{corollary}
\begin{proof} For any $1\leq r\leq d_o$, the forgetful morphism $\Parbun_L^{[r]} (d)\to \Parbun_L^{[r-1]} (d)$ clearly induces an isomorphism
\begin{align*} H^0(\Parbun_{L}^{[r-1]} (d+(r-1)\omega_P(\mu_P)), &\mathcal{N}((\vec{\lambda}, [2g^*m\mu_P^*]^{r-1}), m)) \simeq\\
 &H^0(\Parbun_{L}^{[r]} (d+(r-1)\omega_P(\mu_P)), \mathcal{N}((\vec{\lambda}, [2g^*m\mu_P^*]^{r-1}, 0), m)).
\end{align*}
Further, by Lemma \ref{bundleiso}, the isomorphism $\tau_{\mu_P}: \Parbun_{L}^{[r]} (d+(r-1)\omega_P(\mu_P))\to \Parbun_{L}^{[r]} (d+r\omega_P(\mu_P))$ induces an isomorphism in cohomology:
\begin{align*}H^0(\Parbun_{L}^{[r]} (d+r\omega_P(\mu_P)), &\mathcal{N}((\vec{\lambda}, [2g^*m\mu_P^*]^{r}), m))
\simeq \\
&H^0(\Parbun_{L}^{[r]} (d+(r-1)\omega_P(\mu_P)), \mathcal{N}((\vec{\lambda}, [2g^*m\mu_P^*]^{r-1}, 0), m)).
\end{align*}
Combining the above two isomorphisms over all $1\leq r\leq d_o$,
we get the isomorphism:
$$H^0(\Parbun_{L} (d), \mathcal{N}(\vec{\lambda}, m)) \simeq
H^0(\Parbun_{L}^{[d_o]} (d+ d_o\omega_P(\mu_P)), \mathcal{N}((\vec{\lambda}, [2g^*m\mu_P^*]^{d_o}), m)).
$$
Now, the corollary follows from Lemma \ref{levitosemisimple}. (Observe that, by the identity \eqref{eq20}, the equation \eqref{scale}
is satisfied for $(\vec{\lambda}, [2g^*m\mu_P^*]^{d_o})$ and degree $d+ d_o\omega_P(\mu_P)$.)
\end{proof}

Consider the decomposition $\frh=\frh'\oplus \mathfrak{z}^o$, where $\mathfrak{z}^o$ (Lie algebra of $Z^o$) is the center of $\fl$  and
$\frh'$ is the Cartan subalgebra of $\fl'$.
\begin{lemma} \label{central} For any $\mu\in Q^\vee$ (where $Q^\vee$ is the coroot lattice of $\frg$), write
$$\mu=\mu'+\mu^o, \,\,\mu'\in \frh', \mu^o\in \mathfrak{z}^o.$$
Then, $\Exp(2\pi i \mu')$ lies in the center of $L'$.
\end{lemma}
\begin{proof} Since $\mu\in Q^\vee$, $\Exp(2\pi i \mu)=1$ in $G$ (and hence in $L$). Moreover,
$$\Exp(2\pi i \mu)=\Exp(2\pi i \mu')\cdot
\Exp(2\pi i \mu^o).$$
 But, $\Exp(2\pi i \mu^o)\in Z^o$ and hence $\Exp(2\pi i \mu')\in Z^o\cap L'$. In particular, $\Exp(2\pi i \mu')$
lies in the center of $L'$.
\end{proof}

\section{Irredundancy of the inequalities in Theorem \ref{main}}
The aim of this section is to prove the following theorem, which is the multiplicative eigen polytope  analogue of Ressayre's result \cite{Re}. The following result for $G=\SL_m$ was proved by Belkale combining the works \cite{BTIFR,BFulton} (see Remark \ref{FultonC}).

\begin{theorem} \label{irredundance} Let $n\geq 2$. The inequalities
$$\mathscr{I}^{P}_{(u_{1},\ldots, u_{n};d)}:\medskip\medskip\medskip\,\,\,\,\,\,\,\,\,\,\sum^{n}_{k=1}
\omega_P(u_k^{-1}\mu_k) \leq d,$$ given by  part (b) of Theorem \ref{main} (as we run through the standard maximal parabolic subgroups $P$, $n$-tuples $(u_1, \dots, u_n)\in (W^P)^n$ and non-negative integers $d$ such that $
\langle \sigma^{P}_{u_{1}},\ldots, \sigma^{P}_{u_{n}} \rangle_{d}^{\qdot_{0}}=1)$  are pairwise distinct (even up to scalar multiples) and
form an irredundant system of inequalities
defining the eigen polytope $\mathscr{C}_{n}$ inside $\mathscr{A}^{n}$, i.e., the hyperplanes given by the equality in
$\mathscr{I}^{P}_{(u_{1},\ldots, u_{n};d)}$ are precisely the (codimension one) facets of the polytope $\mathscr{C}_{n}$ which intersect the interior of $\mathscr{A}^{n}$.
\end{theorem}
We divide the proof into several parts.
\subsection{}
First of all, the inequalities $\mathscr{I}^{P}_{(u_{1},\ldots, u_{n};d)}$ are pairwise distinct, even up to scalar multiples:

The stabilizer of $\omega_P$ under the action of $W$  is
precisely equal to the subgroup $W_{P}$.  Let $u_k\omega_P=zv_k\omega_P$, for all $1\leq k\leq n$ and some real number $z$
(independent of $k$) and elements $u_k,v_k\in W^P$. By considering the lengths, we see that  $z=\pm 1$. Further, $z\neq -1$, for otherwise $\mathscr{C}_{n}$ would
 satisfy two inequalities with opposite signs contradicting the fact that $\mathscr{C}_{n}$ has non-empty interior in $\frh^n$. Thus, $z=1$ and each
$u_k=v_k$.

 Now, take two standard maximal parabolic subgroups $P\neq Q$ and assume that
$u_k\omega_P=zv_k\omega_Q$, for all $1\leq k\leq n$ and $ d=zd'$, for  some real number $z$
(independent of $k$) and elements $u_k\in W^P, v_k\in W^Q$ and non-negative integers $d,d'$. For $z< 0$, again
$\mathscr{C}_{n}$ would
 satisfy two inequalities with opposite signs (which is not possible). So $z>0$. Now, the only dominant element in the $W$-orbit of a dominant weight
$\lambda$ is $\lambda$ itself. Hence, we get $\omega_P=z\omega_Q$, which is not possible since $P\neq Q$.

Finally, since none of $u_k\omega_P=0$ and $n\geq 2$, we get that the facet determined by any $\mathscr{I}^{P}_{(u_{1},\ldots, u_{n};d)}$
can not be a facet of the alcove $\mathscr{A}^{n}$.
\subsection{}
Now, we show that none of the inequalities $\mathscr{I}^{P}_{(u_{1},\ldots, u_{n};d)}$ can be dropped. In the rest of the proof, we fix a  standard maximal parabolic
subgroup $P$ of $G$,  $u_{1},\dots,u_{n}\in W^{P}$, and $d\geq
0$ such that
\begin{equation}\label{eqn8.1'}
\langle \sigma^{P}_{u_{1}},\ldots, \sigma^{P}_{u_{n}} \rangle_{d}^{\qdot_{0}}=1.
\end{equation}
We wish to show that the inequality
$\mathscr{I}^{P}_{(u_{1},\ldots, u_{n};d)}$
can not be dropped.
We will produce a collection of points  of
$\mathscr{C}_{n}$ for which  the above inequality is  an equality, and such that their convex span has the dimension of
a facet (i.e., $-1+n\dim \frh$). Before we come to its proof, we need some preparatory material.

As in Subsection \ref{space}, for any  quasi-parabolic principal $G$-bundle $\tilde{\mathcal{E}}=(\mathcal{E};\bar{g}_1,\dots,\bar{g}_n)$ over $\pone$,
 define the subscheme
$$Z_d'(\tilde{\mathcal{E}})=\{f\in Z_d(\mathcal{E}): f(b_k) \,\,\,\text{and}\,\,\bar{g}_k\,\,\,\text{are in relative position}\,\,\,u_k \,\,\forall
1\leq k\leq n\},$$ and its open subscheme:
$$Z_d^o(\tilde{\mathcal{E}})=\{f\in Z'_d(\mathcal{E}): \theta(\tilde{\mathcal{P}}(f))\neq 0\},$$
where
$ Z_d(\mathcal{E})$ is the space of sections of  $\mathcal{E}/P$ of degree $d$, $\tilde{\mathcal{P}}(f)$ is the $P$-subbundle of $\mathcal{E}$ associated to $f$ with its parabolic structures $(\bar{p}_1, \dots, \bar{p}_n)$ as described in Subsection \ref{space} and $\theta\in H^0(\Parbun_P(d), \mathcal{R})$ is the section defined in Subsection \ref{trans}.
This gives rise to a stack
$$\pi':\mathcal{Z}'_d \to \Parbun_G \,\,\,\text{with fiber}\,\,  Z'_d(\tilde{\mathcal{E}})\,\,\text{over}\,\,\tilde{\mathcal{E}} ,$$
and an open substack
$$\pi^o:\mathcal{Z}^o_d \to \Parbun_G \,\,\,\text{with fiber}\,\,  Z^o_d(\tilde{\mathcal{E}})\,\,\text{over}\,\,\tilde{\mathcal{E}}.$$
The image of $\pi^o$ is an open substack $\mathcal{V}$ of $\Parbun_G$. Moreover, due to the equation \eqref{eqn8.1'},
$\pi^o:\mathcal{Z}^o_d \to \mathcal{V}$ is an isomorphism of stacks (following the argument in the proof of Lemma \ref{wash}).

Define a morphism of stacks $i: \Parbun_L(d) \to \Parbun_G$,
taking $\tilde{\mathcal{L}}=(\mathcal{L};{l}_1B_L,\dots,{l}_nB_L) \mapsto (\mathcal{L}\times^LG;{l}_1u_1^{-1}B,\dots, {l}_nu_n^{-1}B)$.  Also, the Levification process Gr $\mathcal{P}$ via the one parameter subgroup $t^{\bar{x}_P}$,  as in Subsection \ref{section3.6},  gives a morphism of stacks
$\xi: \mathcal{Z}'_d\to \Parbun_L(d)$. Similarly, consider the  morphism of stacks $j:\Parbun_L(d) \to \mathcal{Z}'_d$, 
$\tilde{\mathcal{L}}=(\mathcal{L};{l}_1B_L,\dots,{l}_nB_L) \mapsto (\mathcal{L}\times^LP;\bar{l}_1,\dots, \bar{l}_n)$. Then, clearly
$$i=\pi'\circ j.$$
These maps are organized in the following diagram:
\begin{equation}
\xymatrix{
\mathcal{Z}^o_d\ar[d]^{\pi^o}\subseteq &\mathcal{Z}'_d \ar[dr]^\xi\ar[d]^{\pi'}  \\
 \mathcal{V}\subseteq & \Parbun_G &   \Parbun_L(d)\ar@/^/[ul]^{j}\ar[l]^i}
\end{equation}

From the definition, it is easy to see that $i^*(\mathcal{M}(\vec{\lambda}_{\vec{u}},m))\simeq \mathcal{N}(\vec{\lambda},m),$ where
$\vec{\lambda}_{\vec{u}}:=(u_1\lambda_1, \dots, u_n\lambda_n).$
\begin{lemma} \label{bundleisomorphism} For any $(\vec{\lambda}, m)$ such that the identity component $Z^o$ of  the center of $L$ acts trivially on $\mathcal{N}(\vec{\lambda},m)$, the line bundles $\xi^*( \mathcal{N}(\vec{\lambda},m))=\xi^*i^*(\mathcal{M}(\vec{\lambda}_{\vec{u}},m))$ and $(\pi')^*( \mathcal{M}(\vec{\lambda}_{\vec{u}},m))$ are isomorphic over the stack $\mathcal{Z}'_d $.

In particular, the lemma applies to any $(\vec{\lambda}, m) \in S$.
\end{lemma}
\begin{proof} Since  $t^{\bar{x}_P}\subset Z^o$  acts trivially on $\mathcal{N}(\vec{\lambda},m)$, the lemma follows from the definition of the Levification map $\xi$ and the following simple lemma.
\end{proof}
Let $\mar$ be a  $\gmm$-equivariant line bundle  on $\Bbb{A}^1$ such that the action of $\gmm$ on $\mar_0$ is trivial. Then, using the
$\gmm$-action and taking limit as $t\to 0$, we get the following:
\begin{lemma}\label{cleff}
There is a canonical identification $\mar_t\leto{\sim}\mar_0$,  for any $t\in \Bbb{A}^1$.
\end{lemma}
Define the subset
\begin{equation} \label{eqn8.1}
S=\{(\vec{\lambda}, m)=(\lambda_1, \dots, \lambda_n;m)\in X(H)^n\times \mathbb{Z}_+: H^0(\Parbun_{L}(d),\mathcal{N}(\vec{\lambda},m)^r)\neq 0,\,\text{ for some}\,\, r>0\}.
\end{equation}
By the equation \eqref{scale}, for any $(\vec{\lambda}, m)\in S$, $\sum_{k=1}^n \lambda_k(x_P)=4g^*dm/\langle \alpha_P, \alpha_P\rangle.$

Define a map $S\to \frh^n$ by
$(\vec{\lambda},m)\mapsto\frac{1}{2g^* m}\vec{\lambda}$.
We will show (in Section \ref{lasts}) that the convex span of the image of this map has dimension $-1+n\dim \frh$.

\begin{proposition} \label{propn8.4} For any $(\vec{\lambda}, m)\in  S$, there exists a divisor $D_{(\vec{\lambda},m)}\subset \Parbun_G$ contained in the complement of
$\mathcal{V}$ such that
$\mathcal{M}(\vec{\lambda}_{\vec{u}},m)^r(D_{(\vec{\lambda}, m)})$ has a nonzero section over $\Parbun_G$ for some $r>0$.

Moreover, the line bundle $\mathcal{O}(D_{(\vec{\lambda},m)})$ over $\Parbun_G$  is of the form  $\mathcal{M}(\vec{\lambda}_o,m_o)$ for
some $m_o\geq 0$ and some  characters
$\vec{\lambda}_o=(\lambda^o_1, \dots, \lambda^o_n)$ with each $\lambda^o_k$ being a $G$-dominant character of level $m_o$.
Further, $i^*(\mathcal{O}(D_{(\vec{\lambda},m)}))$ admits a nonzero section over $\Parbun_L(d)$.

For a finite subcollection $F$ of $(\vec{\lambda}, m)\in  S$, we can choose a line bundle $\mathcal{O}(D_F)$ as above, which works for all $\vec{\lambda}\in  F$.
\end{proposition}
\begin{proof} By assumption, $\mathcal{N}(\vec{\lambda},m)^r$ has a nonzero section over $\Parbun_L(d)$. Since $(u_1, \dots, u_n;d)$ is quantum Levi-movable, $j^{-1}(\mathcal{Z}^o_d)$ is a non-empty open subset of $\Parbun_L(d)$ and hence $\xi^*(\mathcal{N}(\vec{\lambda},m)^r)$ has  a nonzero section over $\mathcal{Z}^o_d$. Thus, by Lemma \ref{bundleisomorphism} and the isomorphism $\mathcal{Z}^o_d\simeq \mathcal{V}$, we get that
$\mathcal{M}(\vec{\lambda}_{\vec{u}},m)^r$ has a nonzero section $\sigma$ over $\mathcal{V}$.  Since $\Parbun_G$ is a smooth Artin stack (see \cite{W} and the references therein), there is a divisor $D_{(\vec{\lambda},m)} \subset \Parbun_G\setminus \mathcal{V}$   such that
$\sigma$ extends as a nonzero section of $\mathcal{M}(\vec{\lambda}_{\vec{u}},m)^r(D_{(\vec{\lambda},m)})$ over the whole of $\Parbun_G$ (use a smooth atlas of $\Parbun_G$).

To prove the second part, by \cite[Proposition 8.7]{LS}, some suitable power of $\mathcal{O}(D_{(\vec{\lambda},m)})$ is of the form $\mathcal{M}(\vec{\lambda}_o, m_o)$ for suitable characters  $\vec{\lambda}_o=
(\lambda^o_1, \dots, \lambda^o_n)$ and integer $m_o$. Since $\mathcal{O}(D_{(\vec{\lambda},m)})$ has a nonzero global section,  we get that $m_o\geq 0$ and moreover each $\lambda^o_k$ is dominant of level $2g^*m_o$ (cf. Lemma \ref{cohvanishing}). Further, since the canonical section of
$\mathcal{O}(D_{(\vec{\lambda},m)})$ does not vanish on $\Parbun_G\setminus D_{(\vec{\lambda},m)} \supset \mathcal{V}$, and
$i^{-1} (\mathcal{V})$ is non-empty, we get that $i^*(\mathcal{O}(D_{(\vec{\lambda},m)}))$ admits a nonzero section over $\Parbun_L(d)$.
This proves the proposition.
\end{proof}
\subsection{}\label{lasts}

We are now ready to complete the proof of Theorem \ref{irredundance}.

\noindent
{\it Proof of Theorem \ref{irredundance} (continued):} Let  $\mathscr{C}_{n}(L')$ denote the eigen polytope of $L'$ (which is simply-connected, semisimple group). For any central element $z$ of $L'$, consider the twisted  eigen polytope
$$
\mathscr{C}_{n}(L')_z := \left\{\vec{\mu'}=(\mu'_{1},\ldots,\mu'_{n})\in \mathscr{A}^{n}_{L'}:
z\in C(\mu'_{1})\dots C(\mu'_{n})\right\},
$$
where $
\mathscr{A}_{L'}\subset \fh'$ is the fundamental alcove of $L'$ (which is the product of the fundamental alcoves of  the simple components of $L'$) and
 $C(\mu'_{k})$ denotes the conjugacy class of $\Exp(2\pi i\mu'_k)$ under a maximal compact subgroup of $L'$.

Since $z$ is central, it is easy to see that (just as $\mathscr{C}_{n}(L')$) $\mathscr{C}_{n}(L')_z$ is a rational convex
polytope with nonempty interior in $\mathfrak{h'}^n$.

Let $\mu_P'$ be the $\frh'$-component of $\mu_P\in \frh=\frh'\oplus \mathfrak{z}^o$, where $\mu_P\in Q^\vee$ is any element satisfying Lemma
\ref{specialmu}. Then, by Lemma \ref{central}, $\Exp(2\pi i\mu'_P)$ is central in $L'$. Thus, taking $z=(\Exp(2\pi i\mu'_P))^{-d_o}$, for any
$\vec{\mu'}\in \mathscr{C}_{n}(L')_z$, we get $(\vec{\mu'}, [\mu'_P]^{d_o})\in \mathscr{C}_{n+d_o}(L')$. (Observe that $\mu'_P\in \mathscr{A}_{L'}.$)

Thus, applying Theorem \ref{evproblem} and Remark \ref{remark5.3} for simple $G$ replaced by semisimple $L'$ and $n$ replaced by $n+d_o$,
we get that for any rational point $\vec{\mu'}\in \mathscr{C}_{n}(L')_z$, there exists a large enough positive integer $r$ divisible by $2g^*$ such that
\begin{equation}\label{eqn8.6}
   H^0(\Parbun_{L'}^{[d_o]},\mathcal{N}'((r(\vec{\mu'})^*, [r(\mu'_P)^*)]^{d_o}), r/2g^*)) \neq 0,
 \end{equation}
where $(\vec{\mu'})^*:=(\kappa^{-1}(\mu'_n), \dots, \kappa^{-1}(\mu'_n)).$
Thus, by Corollary \ref{cor7.6}, we see (possibly by taking a multiple of $r$) that
$$ H^0(\Parbun_L(d),  \mathcal{N}(r(\vec{\mu})^*, r/2g^*))  \neq 0,$$
for any rational $\vec{\mu}=(\mu_1, \dots, \mu_n)\in \frh^n$ satisfying the following two conditions:

(a) The  $\frh'$-component of $\mu_k$ coincides with $\mu'_k$, for all $1\leq k\leq n$, and

(b)  $r (\vec{\mu})^*$ satisfies the condition \eqref{scale} for $m=r/2g^*$, i.e.,
\begin{equation} \label{eqn8.7} \sum_{k=1}^n \langle \mu_k , x_P \rangle =2d/\langle \alpha_P, \alpha_P\rangle.
\end{equation}
Thus, for any rational point $\vec{\mu'}\in \mathscr{C}_{n}(L')_z$, we get that there exists a large enough positive integer $r$ such that
$(r (\vec{\mu})^*, r/2g^*)\in S$, for any rational   $\vec{\mu}$ satisfying the above conditions (a) and (b). Take a finite collection $\mathcal{F}$ of such
$\vec{\mu}$  such that their convex span is of dimension $-1+n\dim \frh$. (This is possible since $\mathscr{C}_{n}(L')_z$ is of dimension
$n\dim \frh'$ and, for any $\vec{\mu'}$, the extension of $\vec{\mu'}$ to $\vec{\mu}$ satisfying the above conditions (a) and (b) is a
$(n-1)$-dimensional space.) Thus, we  can find a uniform positive integer $r_o$ such that $(r_o(\vec{\mu})^*, r_o/2g^*)\in S$ for any
$\vec{\mu}\in \mathcal{F}$. Thus, by Proposition \ref{propn8.4}, we get that there exists $(\vec{\lambda}_o, m_o)$ such that (replacing $r_o$ by a positive multiple of $r_o$)
$$H^0(\Parbun_G,  \mathcal{M}(r_o(\vec{\mu})^*_{\vec{u}}+\vec{\lambda}_o, \frac{r_o}{2g^*}+m_o))  \neq 0,$$
and hence so is
$$H^0(\Parbun_G,  \mathcal{M}(r_o(\vec{\mu})^*_{\vec{u}}+r_o\vec{\lambda}_o, \frac{r_o}{2g^*}+r_om_o))  \neq 0.$$
Hence, by Theorem \ref{evproblem} and Lemma \ref{cohvanishing}, $$\frac{\vec{\mu}_{\vec{u}}+\vec{\lambda}_o^*}{1+2g^*m_o}\in \mathscr{C}_{n}$$
 and so is $\vec{\lambda}_o^*/2g^*m_o\in \mathscr{C}_{n}$ (by Proposition \ref{propn8.4}).

Now,
\begin{align*} \sum_{k=1}^n\omega_P(u_k^{-1}(u_k\mu_k))=\sum_{k=1}^n\omega_P(\mu_k)&=\frac{\langle\alpha_P,\alpha_P\rangle}{2}\sum_{k=1}^n
\langle\mu_k, x_P\rangle,\,\,\,\text{by the identity \eqref{eq20}},\\
&=d, \,\,\,\text{by the identity \eqref{eqn8.7}}.
\end{align*}
Thus, for any $\vec{\mu}\in \mathcal{F}$, the element $\vec{\mu}_{\vec{u}}=(u_1\mu_1, \dots, u_n\mu_n)$ lies in the hyperplane $H$ given by the equality of $\mathscr{I}^{P}_{(u_{1},\ldots, u_{n};d)}$. Also, since $i^*( \mathcal{M}(\vec{\lambda}_o,m_o))$ admits a nonzero section over
$\Parbun_L(d)$ (by Proposition \ref{propn8.4}), we get that (by the equation \eqref{scale}),
$$\sum_k(u_k^{-1}\lambda_k^o)(x_P)=\frac{4g^*dm_o}{\langle \alpha_P, \alpha_P\rangle},
$$
i.e., $\vec{\lambda}_o^*/2g^*m_o$ lies in $H$. Hence,
the convex combination
$$\frac{1}{1+2g^*m_o}\vec{\mu}_{\vec{u}} + \frac{2g^*m_o}{1+2g^*m_o}\frac{\vec{\lambda}_o^*}{2g^*m_o}=\frac{\vec{\mu}_{\vec{u}} +\vec{\lambda}_o^*}{1+2g^*m_o}\in H\cap  \mathscr{C}_{n}.$$
From this we see that $ H\cap  \mathscr{C}_{n}$ is of dimension $-1+n\dim \frh$, since, by assumption, convex span of $\mu\in \mathcal{F}$ is of this dimension. This proves Theorem \ref{irredundance} completely.

\begin{remark}\label{FultonC} Let $G=\SL_r$ and
$\vec{\lambda}=(\lambda_1,\dots,\lambda_n)$ be a collection of dominant integral weights each of level
$\leq k$ (which is arbitrary, and not necessarily a multiple of $g^*=r$) such that their sum lies in the root lattice of $\frg$, the Lie algebra of $\SL_r$. Let ${V}_{\frg,\vec{\lambda},k}$  be
the space of conformal blocks for $\vec{\lambda}$ at level $k$ for the marked curve $(\pone,b_1,\dots,b_n)$  (see, e.g., \cite{Beauville} for the definition of conformal blocks). The quantum generalization of Fulton's conjecture (abbreviated as QFC) asserts that for any positive integer $N$,   ${V}_{\frg,\vec{\lambda},k}$ has rank $1$ if and only if ${V}_{\frg,N\vec{\lambda},Nk}$ has rank $1$. As observed in \cite[Page 50]{BTIFR}, QFC for $\SL_r$ is implied by the  irredundance (i.e., Theorem \ref{irredundance}) for $\SL_m$ for each $m>r$. Therefore, one obtains a new proof of QFC. 

The article \cite{BTIFR} shows also that Theorem \ref{irredundance} holds for $\SL_m$ assuming the validity of QFC. In \cite{BFulton}, a geometric proof of the classical Fulton conjecture (proved earlier by Knutson-Tao-Woodward \cite{KTW}) was given. It was noted there
that the proof carries over to the quantum case as well (but full proofs were not given in the quantum case).
\end{remark}

\section{Example: Determination of deformed product in quantum cohomology for rank-$2$ groups}

We determine the deformed product in the quantum cohomology for $G/P$, where $G$ is a rank-$2$ group and $P$ is a maximal parabolic subgroup which is not cominuscule. The maximal parabolic $P_i$ refers the one with $S_{P_i}=\{\alpha_i\}$.
In the following examples, we follow the indexing convention as in Bourbaki \cite[Planche II, IX]{B}, and the classes $a_i,b_i,c_i$ refer to the unique classes $\sigma^P_{u_i}
\in H^{2i}(X^P, \mathbb{Z})$, i.e., corresponding to the Schubert varieties $X^P_{u_i}$  of codimension $i$. As in Definition \ref{qdeformquantum},
the variable $q$ (resp. $\tau$) refers to the quantum (resp. deformed) variable.

{\bf Example 1. $G=B_2, P=P_2:$}
\begin{center}
\begin{tabular} {|c|c|c|c|c|c|}
\hline

$H^*(G/P_2)$ & $a_0$& $a_1$ &  $a_2$ &  $a_3$   \\
\hline $a_0$ &$a_0$& $a_1$ &  $a_2$ &  $a_3$   \\ \hline
 $a_1$ & &$ {\defpar} a_2$ & $ a_3$ & $\tau qa_0$ \\ \hline
 $a_2$ & & &$ qa_0$ & $ qa_1$  \\ \hline
 $a_3$ & & & & $ \tau qa_2$  \\ \hline
\end{tabular}
\end{center}

{\bf Example 3. $G=G_2, P=P_1:$}

\begin{center}
\begin{tabular} {|c|c|c|c|c|c|c|}
\hline

$H^*(G/P_1)$ & $b_0$&  $b_1$ &  $b_2$ &  $b_3$ &  $b_4$ & $b_5$   \\
\hline $b_0$ &$b_0$&  $b_1$ &  $b_2$ &  $b_3$ &  $b_4$ & $b_5$   \\ \hline
 $b_1$ & & $ {\defpar}^2 b_2$ & $5{\defpar} b_3$ & ${\defpar}^2 b_4$ &  $b_5+\tau^2qb_0$
& $\tau^2qb_1$\\ \hline
$b_2$ & & &$5 {\defpar}b_4$ & $b_5+\tau^2qb_0$ &  $2qb_1$ & $\tau^2qb_2$ \\ \hline
$b_3$ & & & &$ \tau qb_1$ & $\tau qb_2$ &  $\tau^2qb_3$  \\ \hline
$b_4$ & & & & &$ 2 qb_3$ & $ \tau^2qb_4$  \\ \hline
$b_5$ & & & & & &$ \tau^4 q^2b_0$   \\ \hline
\end{tabular}
\end{center}

{\bf Example 4. $G=G_2, P=P_2:$}

\begin{center}
\begin{tabular} {|c|c|c|c|c|c|c|c|}
\hline

$H^*(G/P_2)$ &$c_0$&  $c_1$ &  $c_2$ &  $c_3$ &  $c_4$ & $c_5$   \\
\hline $c_0$ &$c_0$&  $c_1$ &  $c_2$ &  $c_3$ &  $c_4$ & $c_5$   \\ \hline
 $c_1$ & & $ 3c_2$ & $2{\defpar} c_3+\tau qc_0$ & $3c_4+ qc_1$ &  $c_5+qc_2$
& $\tau qc_3+2\tau q^2c_0$\\ \hline
$c_2$ & & & $2 {\defpar}c_4+\tau qc_1$ & $c_5+2q c_2$ &  $\tau qc_3+\tau q^2c_0$ & $\tau qc_4+\tau q^2c_1$ \\ \hline
$c_3$ & & & &$ 2 qc_3+2q^2c_0$ & $qc_4+q^2c_1$ &  $2q^2c_2$  \\ \hline
$c_4$ & & & & & $ q^2c_2$ & $ \tau q^2c_3$  \\ \hline
$c_5$ & & & & & &$ 2\tau q^2c_4$   \\ \hline
\end{tabular}
\end{center}

\end{document}